\documentclass[11pt,twoside,a4paper]{article}
\usepackage{amsmath,amssymb,mathrsfs}
\usepackage{amsthm}

\usepackage{inputenc}
\usepackage{avant}
\usepackage[english]{babel}
\usepackage{thmtools,thm-restate}
\usepackage[nottoc]
{tocbibind}
\usepackage{graphicx}
\usepackage{enumitem}
\usepackage{tikz}
\usepackage[stretch=10]{microtype}
\usepackage[plain]{fullpage}

\usepackage{bbm}
\usepackage{esint}
\usepackage{authblk}
\usepackage{mathtools}
\usepackage{dsfont}
\usepackage{interval}
\intervalconfig{soft open fences}

\usepackage{nicematrix}
\NiceMatrixOptions{cell-space-limits = 1pt}

\usepackage{xcolor}
\usepackage[inner=2.4cm,outer=2.4cm,top=2.8cm,bottom=2.8cm]{geometry}
\usepackage[all]{xy}
 \usepackage{todonotes}

\usepackage[bookmarks=true]{hyperref}
\usepackage{xcolor}
\hypersetup{
    colorlinks,
    linkcolor={red!50!black},
    citecolor={blue!50!black},
    urlcolor={blue!80!black},
}

\mathtoolsset{showonlyrefs}  

\DeclareMathOperator{\id}{id}

\DeclareMathOperator*{\essinf}{ess\,inf}
\newcommand{\norm}[1]{\left\lVert#1\right\rVert}
\newcommand{\normdot}{{\|\!\cdot\!\|}}
\newcommand*\diff{\mathop{}\!\mathrm{d}}

\newcommand{\Leb}{\mathscr{L}}

\newcommand{\N}{\mathbb{N}}
\newcommand{\R}{\mathbb{R}}

\newcommand{\p}{\mathtt p} 
\newcommand{\de}{\ensuremath{\, \mathrm d}} 

\newcommand\restr[2]{{
  \left.\kern-\nulldelimiterspace 
  #1 
  \right|_{#2} 
  }}

\DeclarePairedDelimiter\abs{\lvert}{\rvert}%

\makeatletter
\let\oldabs\abs
\def\abs{\@ifstar{\oldabs}{\oldabs*}}
\newcommand{\cd}{\mathsf{CD}}

\newcommand{\MCP}{\mathsf{MCP}}

\newcommand{\X}{\mathsf{X}}

\newcommand{\J}{\mathcal{J}}

\newcommand{\Lip}{\mathsf {Lip}}

\newcommand{\di}
{\mathsf d} 
\newcommand{\m}{\mathfrak m} 

\DeclareMathOperator{\Geo}{Geo}

\newcommand{\Prob}{\mathscr{P}}

\newcommand{\dis}{\mathcal D}

\newcommand{\sF}{sub-Fin\-sler }

\newcommand{\sinom}{\sin_\Omega}
\newcommand{\cosom}{\cos_\Omega}
\newcommand{\sinomp}{\sin_{\Omega^\circ}}
\newcommand{\cosomp}{\cos_{\Omega^\circ}}

\newcommand{\e}{{\rm e}}

\newcommand{\hei}{\mathbb{H}}

\renewcommand{\subset}{\subseteq}
\renewcommand{\epsilon}{\varepsilon}
\renewcommand{\phi}{\varphi}

\usepackage{titlesec}
\titleformat{\section}
  {\scshape\large\centering}
  {\thesection}{0.5em}{}
\titleformat{\subsection}
  {\scshape\centering}
  {\thesubsection}{0.4em}
{}

\title{\textbf{\normalsize{\uppercase{Curvature exponent of sub-Finsler Heisenberg groups}}}}
\date{\today}

\author{Samu\"el Borza\footnote{Faculty of Mathematics, University of Vienna. \textit{E-mail}:  \href{mailto:samuel.borza@univie.ac.at}{samuel.borza@univie.ac.at}}, \ 
Mattia Magnabosco\footnote{Mathematical Institute, University of Oxford. \textit{E-mail}:  \href{mailto:mattia.magnabosco@maths.ox.ac.uk}{mattia.magnabosco@maths.ox.ac.uk}}, \ 
Tommaso Rossi\footnote{Laboratoire Jacques-Louis Lions, Sorbonne Universit\'e. \textit{E-mail}: \href{mailto:tommaso.rossi@inria.fr}{tommaso.rossi@inria.fr}} \ 
and Kenshiro Tashiro\footnote{Okinawa Institute of Science and Technology (OIST). \textit{E-mail}:  \href{mailto:kenshiro.tashiro@oist.jp}{kenshiro.tashiro@oist.jp}}}

\newtheoremstyle{remark}
        {10pt}
        {10pt}
        {}
        {}
        {\itshape}
        {.}
        {.4em}
        {}

\newtheoremstyle{proof}
        {10pt}
        {10pt}
        {}
        {}
        {\itshape}
        {.}
        {.4em}
        {}
        
\newtheoremstyle{definition}
        {10pt}
        {10pt}
        {}
        {}
        {\bfseries}
        {.}
        {.4em}
        {}

\newtheoremstyle{theorem}
        {10pt}
        {10pt}
        {\slshape}
        {}
        {\bfseries}
        {.}
        {.4em}
        {}

\theoremstyle{theorem}

\newtheorem{theorem}{Theorem}[section]
\newtheorem{prop}[theorem]{Proposition}
\newtheorem{corollary}[theorem]{Corollary}
\newtheorem{lemma}[theorem]{Lemma}

\theoremstyle{definition}
\newtheorem{definition}[theorem]{Definition}

\theoremstyle{remark}
\newtheorem{remark}[theorem]{Remark}

\theoremstyle{proof}
\newtheorem*{pro}{Proof}
 {\popQED\end{pro}}

\makeatletter
\renewcommand\xleftrightarrow[2][]{%
  \ext@arrow 9999{\longleftrightarrowfill@}{#1}{#2}}
\newcommand\longleftrightarrowfill@{%
  \arrowfill@\leftarrow\relbar\rightarrow}
\makeatother

\makeindex
\begin{document}

\maketitle

\begin{abstract} 
    The curvature exponent $N_{\mathrm{curv}}$ of a metric measure space is the smallest number $N$ for which the measure contraction property $\MCP(0,N)$ holds. In this paper, we study the curvature exponent of sub-Finsler Heisenberg groups equipped with the Lebesgue measure. We prove that $N_{\mathrm{curv}} \geq 5$, and  the equality holds if and only if the corresponding sub-Finsler Heisenberg group is actually sub-Riemannian.  Furthermore, we show that for every $N\geq 5$, there is a sub-Finsler structure on the Heisenberg group such that $N_{\mathrm{curv}}=N$.
\end{abstract}


\section{Introduction}

It is well-known that, in Riemannian geometry, many important results are a consequence of an assumption on the curvature of the space. For instance, a lower bound on the Ricci curvature tensor, i.e. $\mathrm{Ric} \geq K$, implies a plethora of fundamental theorems, such as the Bonnet-Myers theorem, Brunn–Minkowski and Prékopa–Leindler inequalities, Poincaré inequalities, and Lévy-Gromov's isoperimetric inequalities, see  \cite{myers,CEMCS,Hebeybook,levy-gromov}. 

In recent decades, with the aim of extending these results to non-Riemannian spaces, a lot of effort has been put into introducing and studying synthetic notions of curvature bounds that generalize the classical ones from Riemannian geometry, starting from the seminal contributions of \cite{sturm2006-1, sturm2006, lott--villani2009}. In these works, using the language of optimal transport, the authors introduced the curvature-dimension condition $\mathsf{CD}(K, N)$, generalizing the Riemannian conditions ``$\mathrm{Ric} \geq K$ and $\mathrm{dim} \leq N$'' to metric measure spaces. Subsequently, Ohta in \cite{ohta2007} introduced the measure contraction property $\mathsf{MCP}(K, N)$, which is another synthetic condition that is weaker than $\mathsf{CD}(K, N)$ and is based on a non-smooth characterization of the Bishop--Gromov inequality. More precisely, in a metric measure space $(\X, \di, \mathfrak{m})$, the $\mathsf{MCP}(0, N)$ condition is, roughly speaking, equivalent to having $\mathfrak{m}(A_{t, x}) \geq t^N \mathfrak{m}(A),$ for every Borel set $A \subseteq \X$, every $x \in \X$ and every $t \in \interval{0}{1}$, where $A_{t, x} := \{\gamma(t) \mid \gamma : \interval{0}{1} \to \X \text{ is a minimizing geodesic}, \gamma(0) = x, \gamma(1) \in A\}$ is the $t$-intermediate set of $A$ from $x$.

Sub-Riemannian and sub-Finsler manifolds are broad generalizations of Riemannian and Finsler manifolds, where a smoothly varying norm is defined only on a subset of preferred directions, called distribution.
Because of this singular structure, in this setting it is not possible to have a theory of Ricci curvature akin to the one in Riemannian geometry. The prototypical example is the three-dimensional Heisenberg group $\mathbb{H}$, which is the Lie group whose Lie algebra is generated by $X$, $Y$, and $Z$, with the only non-zero bracket being $[X, Y] = Z$. Fixing an inner product (resp. a norm) on $\mathcal{D} := \mathrm{span}\{X, Y\} \cong \mathbb{R}^2$ induces a sub-Riemannian (resp. sub-Finsler) structure on $\mathbb{H}$. As will be made clearer in the next sections, the geometry of these spaces, the shape of geodesics, and their regularity are tightly tied to the choice of the sub-Finsler norm. In this context, there is an intrinsic reference measure namely its Haar measure, which is proportional to the Lebesgue measure $\Leb^3$. This defines a metric measure space and it is therefore natural to ask whether some synthetic curvature-dimension conditions, such as the $\mathsf{CD}$ or $\mathsf{MCP}$ conditions, could hold in the sub-Riemannian and sub-Finsler Heisenberg groups, or more generally, in \sF manifolds (equipped with a smooth measure).

The validity of synthetic curvature-dimension conditions in the sub-Riemannian Heisenberg group was first studied in \cite{juillet2009}, where it was shown that in this metric measure space no $\mathsf{CD}(K,N)$ condition can hold and that $\mathsf{MCP}(K, N)$ holds if and only if $K \leq 0$ and $N \geq 5$. The failure of the $\cd(K,N)$ condition was then proven for every sub-Riemannian manifold equipped with a positive smooth measure, see \cite{juillet2020,magnaboscorossi,rizzistefani}. In a series of works \cite{borzatashiro, magnabosco2023failure, borza2024measure}, we have initiated the study of these synthetic curvature-dimension bounds in sub-Finsler geometry. We found that, in a sub-Finsler Heisenberg group, no $\mathsf{CD}(K, N)$ condition can hold, whatever norm $\normdot$ is fixed on $\mathcal{D}$. Furthermore, $\mathsf{MCP} (K, N)$ would also fail for every $K$ and $N$ if $\normdot$ is neither $C^1$ nor strongly convex. However, if $\normdot$ is $C^{1, 1}$ and strongly convex, the $\mathsf{MCP}(K, N)$ condition would hold for some constants $K$ and $N$. We have also highlighted mixed behaviours when $\normdot$ is $C^1$ and strongly convex but not $C^{1, 1}$.

In this paper, we study the \emph{curvature exponent} of \sF Heisenberg groups. This is defined as the smallest $N \in \ointerval{1}{+\infty}$ such that $\mathsf{MCP}(0, N)$ is satisfied. Note that, because of the homogeneous structure of $\mathbb{H}$, it is actually enough to focus on $\mathsf{MCP}(0, N)$. In metric measure spaces, a lower bound of the curvature exponent is given by the geodesic dimension, see \cite[Thm.\ 4.20]{barilari2022unified}. The work of Juillet \cite{juillet2009} shows that the curvature exponent of the sub-Riemannian Heisenberg group is $N_{\mathrm{curv}} = 5$ and coincides with its geodesic dimension. We have conjectured in \cite[Conjecture 1.5]{borza2024measure} that among all the sub-Finsler Heisenberg groups, only the sub-Riemannian one has $N_{\mathrm{curv}} = 5$. This \textit{rigidity} result is proven in Section \ref{sec:rigidity} (see Theorem \ref{thm:rigidityinsection}) of the present work.

\begin{theorem}[Rigidity of the curvature exponent]
\label{thm:rigidity}
    Any Heisenberg group $\hei$, equipped with a sub-Finsler norm $\normdot$ and the Lebesgue measure $\Leb^3$, satisfies $N_{\mathrm{curv}} \geq 5$. Furthermore, $N_{\mathrm{curv}} = 5$ if and only if $\normdot$ is induced by a scalar product, i.e., the sub-Finsler Heisenberg group $\hei$ is actually sub-Riemannian.
\end{theorem}
The previous theorem shows the remarkable fact that it is possible to single out the sub-Riemannian Heisenberg group among all the sub-Finsler ones solely by examining its curvature exponent. In Section \ref{sec:continuity}, we introduce the topology of $C^k$-strong convergence on the set of norms and we prove that the curvature exponent is \emph{continuous} with respect to this topology (see Theorem \ref{thm:continuity}). Using this result, we are able to build a norm on $\mathcal D$, such that the associated \sF Heisenberg group has a prescribed curvature exponent (see Theorem \ref{thm:continuouscurvexpSMOOOOOOTH}).

\begin{theorem}[Sub-Finsler Heisenber group with prescribed curvature exponent]
\label{thm:INTROcontinuity}
    For any $N \in [5,+\infty)$, there exists a $C^\infty$ and strongly convex norm $\normdot$ such that the associated sub-Finsler Heisenberg group $(\hei, \di, \Leb^3)$ has curvature exponent $N_{\mathrm{curv}} = N$.
\end{theorem}
As a consequence, we show that for every $N \in [5,+\infty)$, there is indeed a metric measure space $(\X, \di, \mathfrak{m})$ having $N_{\mathrm{curv}} = N$. These results and their proofs are based on the analysis of metric measure spaces, convex geometry and other tools developed in \cite{borza2024measure}, see Section \ref{sec:preliminaries}. We can expect that similar conclusions will also be observed in other sub-Finsler Carnot groups and, more generally, in sub-Finsler geometry.

\paragraph{Strategy of the proof of Theorem \ref{thm:rigidity}.} We outline here the main idea behind the proof of our rigidity result, inspecting the case of a sub-Finsler Heisenberg group equipped with a \emph{smooth} and strongly convex norm $\normdot$. In a \sF Heisenberg group, the infinitesimal volume contraction of the $t$-intermediate set is controlled by the Jacobian of the exponential map $\mathcal{J}_t(r,\phi,\omega)$, see Definition \ref{def:exponentialmap} and Proposition \ref{prop:NONO}. Hence, the measure contraction property and the curvature exponent can be characterized in terms of this function, cf.\ Section \ref{sec:preliminaries} for more details. In particular, the curvature exponent $N_{\mathrm{curv}}$ is the smallest $N\in [5,\infty)$ such that 
\begin{equation}
\label{eq:mcp_ineq}
    \mathcal{J}_t(r,\phi,\omega) \geq t^{N}\mathcal{J}_1(r,\phi,\omega),\qquad\forall\, t\in [0,1],
\end{equation}
for all $(r,\phi,\omega)$ in the domain of the Jacobian. In the case where the norm is strongly convex and smooth, the Jacobian is a smooth map, so that performing a Taylor expansion around $\omega=0$, we obtain: 
\begin{equation*}\label{eq:introtaylor}
    \mathcal{J}_t(r,\phi,\omega)=\frac{r^3t^5C_\circ'(\phi)}{12}(C_\circ'(\phi) +C_\circ''(\phi)\omega t+o(\omega)),\qquad\text{as }\omega\to 0, \ \forall\,t\in [0,1],
\end{equation*}
where $o(\omega)$ is a uniform reminder and $C_\circ$ is the angle correspondence map, cf.\ \ref{def:correspondence}. Then,  \eqref{eq:mcp_ineq} is equivalent to asking that $N_{\mathrm{curv}}$ is the smallest $N\in [5,\infty)$ such that
\begin{equation}
\label{eq:introineq}
    (C_\circ'(\phi) +C_\circ''(\phi)\omega t+o(\omega))\geq t^{N-5}(C_\circ'(\phi) +C_\circ''(\phi)\omega+o(\omega)),\qquad \forall\,t\in [0,1].
\end{equation}
From this inequality, we easily see that $N_{\mathrm{curv}}=5$ if and only if $C_\circ''\equiv 0$, or equivalently that $C_\circ$ is affine. But, according to Proposition \ref{prop:scalar}, $C_\circ$ is affine if and only if $\normdot$ is induced by a scalar product and we conclude the proof. 

In the general case, where the norm is not smooth, we can not perform a Taylor expansion as in \eqref{eq:introtaylor}. Nonetheless, if $\normdot$ is not induced by a scalar product, we are still able to find an angle $\varphi$ where $C_\circ''(\varphi)>0$ in \emph{a weak sense} (see Lemma \ref{lemma:D2>H}) and exploit this behaviour to show that $N_{\mathrm{curv}}>5$. In particular, in Lemmas \ref{lemma:Deltat<=} and \ref{lemma:<=Deltat<=}, we provide crucial estimates which replace the Taylor expansion of $C_\circ$ in the non-smooth setting. These estimates, in turn, allow to reproduce an inequality akin to \eqref{eq:introineq} and finalize the proof.

\subsection*{Acknowledgments} 
We thank Prof. Sturm for raising the question of the existence of a metric measure space with fractional curvature exponent. This project has received funding from the European Research Council (ERC) under the European Union’s Horizon 2020 research and innovation program (grant agreement No. 945655). M.M. acknowledges support from the Royal Society through the Newton International Fellowship (award number: NIF$\backslash$R1$\backslash$231659). T.R. acknowledges from the ANR-DFG project ``CoRoMo'' (ANR-22-CE92-0077-01). K.T. is partially supported by JSPS KAKENHI grant numbers 18K03298, 19H01786, 23K03104.

\section{Preliminaries}
\label{sec:preliminaries}
\subsection{The \texorpdfstring{$\MCP(K,N)$}{MCP(K,N)} condition and curvature exponent} \label{sec:CD}

A metric measure space is a triple $(\X,\di,\m)$ where $(\X,\di)$ is a complete and separable metric space and $\m$ is a locally finite Borel measure on it. In the following, we denote by $C([0, 1], \X)$ the space of continuous curves from $[0, 1]$ to $\X$. For every $t \in [0, 1]$ we call $e_t \colon C([0, 1], \X) \to \X$ the evaluation map, i.e. $e_t(\gamma) := \gamma(t)$. A curve $\gamma\in C([0, 1], \X)$ is said to be a \textit{geodesic} if 
\begin{equation}
    \di(\gamma(s), \gamma(t)) = |t-s| \cdot  \di(\gamma(0), \gamma(1)) \quad \text{for every }s,t\in[0,1].
\end{equation}
We denote by $\Geo(\X)$ the space of all geodesics on $(\X,\di)$. The metric space $(\X,\di)$ is said to be geodesic if every pair of points $x,y \in \X$ can be connected with a curve $\gamma\in \Geo(\X)$.
We denote by $\Prob(\X)$ the set of Borel probability measures on $\X$ and by $\Prob_2(\X) \subset \Prob(\X)$ the set of those having finite second moment. We endow the space $\Prob_2(\X)$ with the Wasserstein distance $W_2$, defined by
\begin{equation}
\label{eq:defW2}
    W_2^2(\mu_0, \mu_1) := \inf_{\pi \in \mathsf{Adm}(\mu_0,\mu_1)}  \int \di^2(x, y) \, \de \pi(x, y),
\end{equation}
where $\mathsf{Adm}(\mu_0, \mu_1)$ is the set of all admissible transport plans between $\mu_0$ and $\mu_1$, namely all the measures $\pi \in \Prob(\X\times\X)$ such that $(\p_1)_\sharp \pi = \mu_0$ and $(\p_2)_\sharp \pi = \mu_1$, where $\p_i$, for $i=1,2
$, is the projection onto the $i$-th factor. The metric space $(\Prob_2(\X),W_2)$ is itself complete and separable, moreover, if $(\X,\di)$ is geodesic, then $(\Prob_2(\X),W_2)$ is geodesic as well. In this case, every geodesic $(\mu_t)_{t\in [0,1]}$ in $(\Prob_2(\X),W_2)$ can be represented with a measure $\eta \in \Prob(\Geo(\X))$, i.e.\ $\mu_t = (e_t)_\# \eta$. 

We present the measure contraction property, or $\MCP(K,N)$ for brevity, firstly introduced by Ohta \cite{ohta2007}. For every $K \in \R$, $N\in (1,\infty)$ and $t\in [0,1]$, the \emph{distortion coefficients} are the functions:

\begin{equation}\label{eq:tau}
    \tau_{K,N}^{(t)}(\theta):=t^{\frac{1}{N}}\left[\sigma_{K, N-1}^{(t)}(\theta)\right]^{1-\frac{1}{N}},\qquad\forall\,\theta\geq 0
\end{equation}
where
\begin{equation}
\sigma_{K,N}^{(t)}(\theta):= 
\begin{cases}

\displaystyle  +\infty & \textrm{if}\  N\pi^{2}\leq K\theta^{2}, \crcr
\displaystyle  \frac{\sin(t\theta\sqrt{K/N})}{\sin(\theta\sqrt{K/N})} & \textrm{if}\  0 < K\theta^{2} < N\pi^{2}, \crcr
t & \textrm{if}\ 
K =0,  \crcr
\displaystyle   \frac{\sinh(t\theta\sqrt{-K/N})}{\sinh(\theta\sqrt{-K/N})} & \textrm{if}\ K < 0.
\end{cases}
\end{equation}

\begin{definition}[$\mathsf{MCP}(K,N)$ condition, \cite{ohta2007}]
\label{def:mcp}
    Given $K\in\R$ and $N\in (1,\infty)$, a metric measure space $(\X,\di,\m)$ is said to satisfy the \emph{measure contraction property} $\mathsf{MCP}(K,N)$ if for every $x\in\text{spt}(\m)$ and every Borel set $A\subset\X$ with $0<\m(A)<\infty$, there exists a $W_2$-geodesic induced by $\eta \in \Prob(\Geo(\X))$ connecting $\delta_x$ and $\frac{\m|_A}{\m(A)}$ such that, for every $t\in[0,1]$,
    \begin{equation}
    \label{eq:mcp_def}
        \frac{1}{\m(A)}\m\geq(e_t)_\#\Big(\tau_{K,N}^{(t)}\big(\di(\gamma(0),\gamma(1))\big)^N\eta(\text{d}\gamma)\Big).
    \end{equation}
\end{definition}

\begin{remark}
\label{rmk:SIUUUUUUU}
Let us recall a useful equivalent formulation of the inequality \eqref{eq:mcp_def}, which holds whenever geodesics are unique, we refer the reader to \cite[Lem.\ 2.3]{ohta2007} for further details. Consider $x\in\text{spt}(\m)$ and a Borel set $A\subset\X$ with $0<\m(A)<\infty$. Assume that for every $y\in A$, there exists a unique geodesic $\gamma_{x,y}:[0,1]\to \X$ joining $x$ and $y$. Then, \eqref{eq:mcp_def} is verified for the marginals $\delta_x$ and $\frac{\m|_A}{\m(A)}$ if and only if 
\begin{equation}
\label{eq:tj_pantaloncini}
    \m\big(A'_{t,x})\big)\geq \int_{ A'}\tau_{K,N}^{(t)}(\di(x,y))^N \de\m(y),\qquad\text{for any Borel set }A'\subset A,
\end{equation}
where $A_{t,x}$ is the $t$-intermediate set defined by, for $A\subset \X$,
    \begin{align}
        A_{t,x} := 
        \{
            y \in \X \, : \, y = \gamma(t)
                \, , \, 
            \gamma \in \Geo(\X)
                \, , \, 
            \gamma(0)=x  
                \ \text{and} \  
            \gamma(1) \in A
        \} 
            \, .
    \end{align}
\end{remark}

On the one hand, it is well known that the measure contraction property has the scaling property:
If $(\X,\di,\m)$ is a $\MCP(K,N)$ space, for every $\alpha,\beta>0$ the scaled space $(\X,\alpha \di,\beta \m)$ is a $\MCP(\alpha^{-2} K,N)$ space.
On the other hand,
a Carnot group (including the Heisenberg group) has the dilation automorphism.
By scaling a metric with the dilation,
we can deduce that a Carnot group satisfies $\MCP(K,N)$ for some $K<0$ if and only if $\MCP(0,N)$ also holds (see for example \cite[Prop.\ 3.17]{borza2024measure}).
Thus finding the optimal number $N$ is one of the central interest of Carnot groups.
Such a number is called the curvature exponent.

\begin{definition}[Curvature exponent] \label{def:curvatureexponent}
    Given a metric measure space  $(\X,\di,\m)$, its \emph{curvature exponent} is defined as 
    \begin{equation}
        N_{\mathrm{curv}}:=\inf\{N\in  (1,+\infty):(\X,\di,\m) \text{ satisfies } \MCP(0,N)\},
    \end{equation}
    where in particular $N_{\mathrm{curv}}=+\infty$ if $(\X,\di,\m)$ does not satisfy $\MCP(0,N)$ for any $N\in (1, \infty)$.
\end{definition}

\begin{remark}
    \label{rmk:min=inf_curv_exp}
    Observe that, for a metric measure space $(\X,\di,\m)$ satisfying $\MCP(0,N)$ for some $N\in\N$, the curvature exponent is a minimum, meaning that $(\X,\di,\m)$ satisfies $\MCP(0,N_\mathrm{curv})$.
\end{remark}

\begin{lemma}
\label{lem:lsc_curv_exp}
    Let $\{(\X,\di_\varepsilon,\m_\varepsilon)\}_{\varepsilon\geq 0}$ a sequence of metric measure spaces. Assume that 
    \begin{equation}
        (\X,\di_\varepsilon,x_\varepsilon,\m_\varepsilon)\xrightarrow[\text{pmGH}]{\varepsilon\to 0}(\X,\di_0,x_0,\m_0).
    \end{equation}
    Then, $N_\mathrm{curv}^0\leq \liminf_{\varepsilon\to 0}N_\mathrm{curv}^\varepsilon$.
\end{lemma}

\begin{proof}
    If the inferior limit of the curvature exponents is infinite, there is nothing to prove. Hence, assume that $\liminf_{\varepsilon\to 0}N_\mathrm{curv}^\varepsilon<\infty$. This means that, for $\varepsilon>0$ sufficiently small, $(\X,\di_\varepsilon,\m_\varepsilon)$ satisfies the $\MCP(0,N_\mathrm{curv}^\varepsilon)$, cf.\ Remark \ref{rmk:min=inf_curv_exp}. Let now $N_0:=\liminf_{\varepsilon\to 0}N_\mathrm{curv}^\varepsilon$ and consider a sequence $\{\varepsilon_n\}_{n\in\N}$ such that $N^{\varepsilon_n}_{\mathrm{curv}}\to N_0$. Then, by stability of the measure contraction property under the pointed measured Gromov--Hausdorff convergence, $(\X,\di_0,\m_0)$ satisfies $\MCP(0,N_0)$. This concludes the proof. 
\end{proof}

\subsection{Convex trigonometry}
\label{sec:convex_trigtrig}
We recall some basic facts about convex trigonometry, see \cite[Sec.\  2.2-2.3]{borza2024measure} for a more comprehensive introduction. Let $\normdot$ be a norm on $\R^2$,
$\Omega\subset \R^2$ the unit ball centered at the identity,
and $\pi_\Omega$ the surface area of $\Omega$.
In this section, we recall the definition of the convex trigonometric functions associated with the norm $\normdot$, firstly introduced in \cite{lok} (see also \cite{lok2}).
Let $f_\Omega:\R^2\to \R$ be a convex function defined by $f_{\Omega}(x):=\frac12\|x\|^2$.
We say that a norm is of class $C^{k,\alpha}$ if the function $f_\Omega$ is of class $C^{k,\alpha}$ on $\R^2\setminus \{(0,0)\}$,
and we say that a norm is strictly convex (resp. strongly convex) if $f_\Omega$ is strictly convex (resp. strongly convex).
Throughout this paper,
for simplicity we assume that a reference norm $\normdot$ is $C^1$ and strictly convex.
Geometrically it is equivalent to the property that $\partial\Omega$ is a $C^1$ curve and $\Omega$ is a strictly convex domain.
\begin{definition}[Convex trigonometric functions]
    For $\theta\in[0,2\pi_{\Omega})$, define $P_\theta$ as the point on the boundary of $\Omega$, such that the area of the sector of $\Omega$ between the $x$-axis $Ox$ and the ray $OP_{\theta}$ is $\frac{1}{2}\theta$ (see Figure \ref{fig:convextrig1}). Moreover, define $\sinom(\theta)$ and $\cosom(\theta)$ as the coordinates of the point $P_\theta$, i.e.
    \begin{equation*}
        P_\theta = \big( \cosom(\theta), \sinom(\theta) \big).
    \end{equation*}
    Finally, extend these trigonometric functions outside the interval $[0,2\pi_{\Omega})$ by periodicity (of period $2 \pi_{\Omega}$), so that for every $k\in \mathbb Z$
    \begin{equation*}
        \cosom(\theta)= \cosom(\theta+2k \pi_{\Omega}), \quad \sinom(\theta)= \sinom(\theta+2k \pi_{\Omega}) \quad \text{and}\quad P_\theta = P_{\theta +2k\pi_{\Omega}}.
    \end{equation*}
\end{definition}
\
\noindent In particular, the maps $P,\sinom,\cosom$ are well-defined on the quotient $\R/ 2 \pi_\Omega \mathbb Z$. Observe that by definition $\sinom(0)=0$ and that when $\Omega$ is the Euclidean unit ball we recover the classical trigonometric functions.

\begin{figure}[ht]
    \begin{minipage}[c]{.47\textwidth}

    \centering
    \begin{tikzpicture}[scale=0.8]

    \draw[white](1.1,3.25)--(0.92142,4.75);
    \fill[color=black!10!white](3,0)--(0,0)--(1.1,3.25)--(3.6,4);
    \fill[white](3,0) .. controls (3.6,2.8) and (2.4,3.8)..(0,3)--(0,4)--(4,4);
    \draw[->] (-4,0)--(4,0);
    \draw[->] (0,-4)--(0,4);
    \draw[very thick] (3,0) .. controls (3.6,2.8) and (2.4,3.8)..(0,3)..controls (-1.2,2.6) and (-2,2.4).. (-2.6,0)..controls (-3.4,-3) and (-2,-3.4).. (0,-3.2)..controls (1.4,-3) and (2.4,-2.4).. (3,0);
    \draw[dotted,blue,thick] (1.1,3.25) --(0,3.25);
    \draw[very thick] (1.1,3.25) --(0,0);
    \draw[very thick] (3,0) --(0,0);
    \draw[very thick,blue](0,3.25)--(0,0);
    \draw[very thick,red](1.1,0)--(0,0);
    \draw[dotted,red,thick] (1.1,3.25) --(1.1,0);
    \filldraw[black] (0,0) circle (1.5pt);
    \filldraw[blue] (0,3.25) circle (1.5pt);
    \filldraw[red] (1.1,0) circle (1.5pt);
    \filldraw[black] (1.1,3.25) circle (1.5pt);

    \node at (-0.9,2)[label=south:${\color{blue}{\sin_\Omega(\theta)}}$] {};
    \node at (0.9,0)[label=south:${\color{red}{\cos_\Omega(\theta)}}$] {};
    \node at (-0.3,0.1)[label=south:$O$] {};
    \node at (2,2)[label=south:$\frac 12 \theta$] {};
    \node at (-2,-1.8)[label=south:$\Omega$] {};
    \node at (1.2,4.2)[label=south:$P_\theta$] {};

    \end{tikzpicture}
    \caption{Values of the generalized trigonometric functions $\cosom$ and $\sinom$.}
    \label{fig:convextrig1}
    
\end{minipage}%
\hfill
\begin{minipage}[h]{.47\textwidth}

    \centering
    \begin{tikzpicture}[scale=0.8]

    \fill[color=black!10!white](3,0)--(0,0)--(1.1,3.25)--(3.6,4);
    \fill[white](3,0) .. controls (3.6,2.8) and (2.4,3.8)..(0,3)--(0,4)--(4,4);
    \draw[->] (-4,0)--(4,0);
    \draw[->] (0,-4)--(0,4);
    \draw[very thick] (3,0) .. controls (3.6,2.8) and (2.4,3.8)..(0,3)..controls (-1.2,2.6) and (-2,2.4).. (-2.6,0)..controls (-3.4,-3) and (-2,-3.4).. (0,-3.2)..controls (1.4,-3) and (2.4,-2.4).. (3,0);
    \draw[very thick] (1.1,3.25) --(0,0);
    \draw[very thick] (3,0) --(0,0);
    \filldraw[black] (0,0) circle (1.5pt);
    \filldraw[black] (1.1,3.25) circle (1.5pt);
    \draw(1.1,3.25)--(-1,3);
    \draw(1.1,3.25)--(3.2,3.5);
    \draw[very thick, ->](1.1,3.25)--(0.92142,4.75);

    \node at (-0.3,0.1)[label=south:$O$] {};
    \node at (2,2)[label=south:$\frac 12 \theta$] {};
    \node at (-2,-1.8)[label=south:$\Omega$] {};
    \node at (1.5,3.3)[label=south:$P_\theta$] {};
    \node at (1.5,4.7)[label=south:$Q_{\phi}$] {};

    \end{tikzpicture}
   \caption{Representation of the correspondence $\theta\xleftrightarrow{\Omega} \phi$.}
    \label{fig:convextrig2}
    
\end{minipage}
\end{figure}

Consider now  the polar set:
\begin{equation*}
    \Omega^\circ := \{p\in \R^2\, :\, \langle p,x\rangle\leq 1 \text{ for every }x\in \Omega\},
\end{equation*}
which is the unit ball of the dual norm $\normdot_\ast$ on $\R^2$, and consider the associated trigonometric functions $\sinomp$ and $\cosomp$. Observe that, by definition of polar set, it holds that 
\begin{equation}
\label{eq:NO_PYTHAGOREAN_IDENTITY_COMETIPERMETTI}
    \cosom(\theta) \cosomp(\phi) + \sinom(\theta)\sinomp(\phi)\leq 1,\qquad \text{for every } \theta,\phi\in \R.
\end{equation}

\begin{definition}[Correspondence]
\label{def:correspondence}
    We say that two angles $\theta,\phi\in \R$ \emph{correspond} to each other and write $\theta \xleftrightarrow{\Omega} \phi$ if the vector $Q_\phi:= (\cosomp(\phi),\sinomp(\phi))$ determines a half-plane containing $\Omega$ (see Figure \ref{fig:convextrig2}).
\end{definition}

By the bipolar theorem \cite[Thm.\ 14.5]{Rockafellar+1970}, it holds that $\Omega^{\circ \circ}=\Omega$.
This fact implies that two angles $\theta,\phi\in \R$ satisfy $\theta\xleftrightarrow{\Omega} \phi$ if and only if $\phi\xleftrightarrow{\Omega^\circ} \theta$.
Moreover,
$\theta\xleftrightarrow{\Omega} \phi$ if and only if the following analogous of the Pythagorean equality holds:
    \begin{equation}\label{eq:pytagorean}
         \cosom(\theta) \cosomp(\phi) + \sinom(\theta)\sinomp(\phi)= 1.
    \end{equation}

\noindent Thanks to the $C^1$ and strictly convexity assumption on $\normdot$,
the correspondence $\theta\xleftrightarrow{\Omega} \phi$ is one-to-one.
Therefore we can define a continuous monotone map $C^\circ$ that maps an angle $\theta$ to the unique angle corresponding to $\theta$ i.e. $\theta\xleftrightarrow{\Omega} C^\circ(\theta)$.
Since the convex set $\Omega$ is symmetric, this function has the following periodicity property:
\begin{equation}
\label{eq:periodicity_Ccirc}
    C^\circ(\theta+\pi_{\Omega} k)=  C^\circ(\theta) +\pi_{\Omega^\circ} k \qquad \text{ for every }k\in \mathbb Z,
\end{equation}
where $\pi_{\Omega^\circ}$ denotes the surface area of $\Omega^\circ$. Analogously, we can define the map $C_\circ$ associated to the correspondence $\phi\xleftrightarrow{\Omega^\circ} \theta$, and it satisfies an analogue of \eqref{eq:periodicity_Ccirc}. Note that $C^\circ\circ C_\circ=C_\circ\circ C^\circ=\mathrm{Id}$ and that the relationship between $C_\circ$ and the reference norm is given, according to \cite[Lem.\ 2.15]{borza2024measure}, by
\begin{equation}
\label{eq:Ccirctonorm}
C_\circ = P^{-1} \circ \diff \normdot_* \circ Q.
\end{equation}

\begin{prop}\label{prop:difftrig}
    Let $\normdot$ be a $C^1$ and strictly convex norm. The associated trigonometric functions $\sinomp$ and $\cosomp$ are differentiable and it holds that
    \begin{equation*}
        \sinomp'(\phi)= \cosom(C_\circ(\phi)) \qquad \text{and}\qquad \cosomp'(\phi)= - \sinom(C_\circ(\phi)).
    \end{equation*}
    Naturally, the analogous result holds for the trigonometric functions $\sinom$ and $\cosom$.
\end{prop}

\subsection{The \sF geometry of the Heisenberg group}
\label{sec:saymyname}

We present here the \sF Heisenberg group and study its geodesics. Let us consider the Lie group $M=\R^3$, equipped with the non-commutative group law, defined by
\begin{equation}
    (x, y, z) \star (x', y', z') = \bigg(x+x',y+y',z+z'+\frac12(xy' - x'y)\bigg),\qquad\forall\,(x, y, z), (x', y', z')\in\R^3,
\end{equation}
with identity element $\e=(0,0,0)$. We define the left-invariant vector fields
\begin{equation}
    X_1:=\partial_x-\frac{y}2\partial_z,\qquad X_2:=\partial_y+\frac{x}2\partial_z.
\end{equation}
The associated distribution of rank $2$ is $\dis:=\text{span}\{X_1,X_2\}$. It can be easily seen that $\dis$ is bracket-generating. Then, letting $\normdot:\R^2\to\R_{\geq 0}$ be a norm, the \emph{\sF Heisenberg group} $\hei$ is the Lie group $M$ equipped with the \sF structure $(\dis,\normdot)$. For further details on \sF geometry, we refer to \cite[Sec.\ 2.2]{magnabosco2023failure}. We define the associated left-invariant norm on $\dis$ as 
\begin{equation}
    \|v\|_\dis:=\|(u_1,u_2)\|,\qquad \text{for every }v=u_1X_1+u_2X_2\in \dis. 
\end{equation}
A curve $\gamma: [0,1]\to \hei$ is \emph{admissible} if its velocity $\dot\gamma(t)$ exists almost everywhere and there exists a function $u=(u_1,u_2)\in L^2([0,1];\R^2)$ such that
\begin{equation}
\label{eq:admissible_curve}
    \dot\gamma(t)= u_1(t)X_1(\gamma(t))+u_2(t)X_2(\gamma(t))\in\dis_{\gamma(t)},\qquad\text{for a.e. }t\in [0,1].
\end{equation}
The function $u$ is called \emph{the control}. We define the \emph{length} of an admissible curve:
\begin{equation}
    \ell(\gamma):=\int_0^1 \norm{\dot\gamma(t)}_{\dis} \de t\in[0,\infty).
\end{equation}
For every couple of points $q_0,q_1\in M$, define the \emph{\sF distance} between them as
\begin{equation*}
    \di (q_0,q_1) := \inf \left\{\ell(\gamma)\, :\, \gamma \text{ admissible, } \gamma(0)=q_0 \text{ and }\gamma(1)=q_1\right\}.
\end{equation*}
We recall that the Chow--Rashevskii Theorem ensures that the \sF distance on $\hei$ is finite, continuous and the induced topology is the manifold one. Note that, since both the norm and the distribution are left-invariant, the left-translations defined by
\begin{equation}
\label{eq:left_translations}
    L_p:\hei\to\hei;\qquad L_p(q):=p\star q,
\end{equation}
are isometries for every $p\in \hei$. 

\begin{remark}
    We assume that the norm is reversible, and this implies that $\di$ is symmetric as well. We believe that the arguments developed here could be adapted to the more general setting of non-reversible \sF structures. However, the measure contraction property is typically defined for classical metric measure spaces and, thus, the distance should be symmetric. To the best of our knowledge, there is no general definition of $\MCP$ for non-symmetric metric measure spaces. For this reason, we restrict ourselves to reversible sub-Finsler norms.
\end{remark}

In the \sF Heisenberg groups, the geodesics were originally studied in \cite{Busemann} and \cite{Bereszynski} for the three-dimensional case and in \cite{lok2} for general left-invariant structures on higher-dimensional Heisenberg groups.
We recall the map $G_t$ which plays the role of a \emph{\sF exponential map} from the origin at time $t$,
thoroughly studied in \cite{borza2024measure}. It can be seen as a generalization of the exponential map of the sub-Riemannian Heisenberg group (see for example \cite{Ambrosio2004}).

\begin{definition}\label{def:exponentialmap}
    Let $\hei$ be the \sF Heisenberg group, equipped with a $C^1$ and strictly convex norm $\normdot$ and let 
    \begin{equation}
        \mathscr{U}:=\R_{>0}\times \R/2\pi_{\Omega^\circ}\mathbb{Z}\times \{(-2\pi_{\Omega^\circ},2\pi_{\Omega^\circ})\setminus\{0\}\}.
    \end{equation} 
    For every $t\in \R$, we define the exponential map at time $t$ as $G_t:\mathscr{U}\to\hei$, such that for any $(r,\phi,\omega)\in \mathscr{U}$, $G_t(r,\phi,\omega):=(x_t,y_t,z_t)$, where
\begin{equation}\label{eq:xyz}
    \begin{cases}
    \begin{aligned}
        x_t(r,\phi,\omega) &= \frac{r}{\omega}\left(\sinomp(\phi+\omega t) - \sinomp(\phi)\right),\\
        y_t(r,\phi,\omega) &= -\frac{r}{\omega}\left(\cosomp(\phi+\omega t) - \cosomp(\phi)\right),\\
        z_t(r,\phi,\omega) &= \frac{r^2}{2\omega^2}\left(\omega t + \cosomp(\phi+\omega t) \sinomp(\phi) - \sinomp(\phi+\omega t ) \cosomp(\phi)\right).
    \end{aligned}
    \end{cases}
\end{equation}
\end{definition}

According to \cite[Prop. 3.5]{borza2024measure},
the mapping $G_t$ holds several properties that characterize the exponential map at time $t$. 
Moreover, as highlighted in \cite{borza2024measure} (see also Propositions \ref{prop:phoenix} and \ref{prop:diffcharacMCP} below), the validity of $\MCP(K,N)$ is related to the properties of the Jacobian of the map $G_t$.
After a straightforward computation,
we can write the Jacobian by using the trigonometric functions.

\begin{prop}
\label{prop:NONO}
    Let $\hei$ be the \sF Heisenberg group, equipped with a $C^1$ and strictly convex norm $\normdot$. The map $G_t:\mathscr{U}\to\hei$ is differentiable with Jacobian 
    \begin{equation}\label{eq:jacobian}
    \begin{split}
         \J_t (r,\phi, \omega) = \frac{r^3 t}{ \omega^4} \bigg[2 &- \Big(\sinomp(\phi + \omega t) \sinom(C_\circ(\phi))+ \cosomp(\phi + \omega t) \cosom(C_\circ(\phi))\Big)  \\
         &- \Big(\sinom\big(C_\circ(\phi + \omega)\big) \sinomp( \phi)+ \cosom\big( C_\circ(\phi + \omega t) \big) \cosomp( \phi)\Big)\\
         &- \omega t \Big(\sinom\big(C_\circ(\phi + \omega t)\big) \cosom(C_\circ(\phi)) - \cosom\big(C_\circ(\phi + \omega)\big) \sinom(C_\circ(\phi)) \Big) \bigg]. 
    \end{split}
    \end{equation}
\end{prop}

On the domain $\mathcal{U}:=\R/2\pi_{\Omega^\circ}\mathbb{Z}\times(-2\pi_{\Omega^\circ},2\pi_{\Omega^\circ})\setminus\{0\}$,
we define the \emph{reduced Jacobian} as the measurable function $\J_R:\mathcal{U}\to\R$, such that for every $(\phi,\psi)\in \mathcal U$, 
\begin{equation}\label{eq:reduced_Jac}
    \begin{split}
         \J_R (\phi, \psi) := 2 &- \Big(\sinomp( \phi +  \psi) \sinom( C_\circ(\phi))+ \cosomp( \phi +  \psi) \cosom( C_\circ(\phi))\Big)  \\
         &- \Big(\sinom\big(C_\circ( \phi +  \psi)\big) \sinomp( \phi)+ \cosom\big( C_\circ( \phi +  \psi) \big) \cosomp( \phi)\Big)\\
         &- \psi \Big(\sinom\big(C_\circ( \phi +  \psi)\big) \cosom( C_\circ(\phi)) - \cosom\big(C_\circ( \phi +  \psi)\big) \sinom( C_\circ(\phi)) \Big). 
    \end{split}
    \end{equation}
In other words,
the reduced Jacobian is defined so that
    $\J_t (r,\phi, \omega) = \frac{r^3 t}{\omega^4} \J_R (\phi, \omega t )$.
The (reduced) Jacobian can be used to study the measure contraction property (see \cite[Sec.\ 3.3]{borza2024measure}).

Recall that according to \cite[Prop.\ 2.16]{borza2024measure}, the correspondence map $C_\circ$ is strictly increasing and Lipschitz if the reference norm $\normdot$ is $C^1$ and strongly convex. The following useful expression for the reduced Jacobian was obtained in \cite{borza2024measure}.

\begin{prop}[{\cite[Prop.\ 5.6. and Eq.\ (63)]{borza2024measure}}]
    \label{prop:formulaJRdJRint}
    Let $\hei$ be the sub-Finsler Heisenberg group, equipped with a $C^1$ and strongly convex norm $\normdot$. Then, its reduced Jacobian $\J_R$ can be expressed in the following way. For all $(\phi,\omega)\in\mathcal{U}$, it holds
\begin{equation}
\label{eq:JRaround0}
\J_R (\phi, \omega) = \underbrace{\frac{1}{2} \int_{\phi}^{\phi + \omega} \left( \int_{\phi}^{\phi + \omega} (t - s)^2 C_\circ'(t) C_\circ'(s) \diff s \right) \diff t}_{=: P(\phi, \omega)} + R(\phi, \omega),
\end{equation}
where
\begin{equation}
    \label{eq:remainderterm}
    \begin{aligned}
    R(\phi, \omega) =&\int_{\phi}^{\phi + \omega} \int_\phi^t \int_\phi^s (t - s)(s - u) \Big[ \sin_{\Omega^\circ}(u) \cos_{\Omega^\circ}(\phi) - \sin_{\Omega^\circ}(\phi) \cos_{\Omega^\circ}(u) \\
    & - (t - \phi) (\cos_{\Omega}(C_\circ(\phi))\cos_{\Omega^\circ}(u) + \sin_{\Omega}(C_\circ(\phi))\sin_{\Omega^\circ}(u)) \Big] C'_\circ(t) C'_\circ(s) C'_\circ(u) \diff u \diff s \diff t.
\end{aligned}
\end{equation}
Furthermore, there exists a constant $M > 0$ such that
    \begin{equation}
    \label{eq:boundremainder}
        |R(\phi,\omega)| \leq M \omega^6, \qquad\text{for every } (\phi,\omega)\in \mathcal{U}.
    \end{equation}
\end{prop}

\begin{remark}
    The bound \eqref{eq:boundremainder} is justified by considering \eqref{eq:remainderterm}, the Lipschitzness of $C_\circ$ (which gives a uniform bound on $C_\circ'$) and the fact that the generalized trigonometric functions are bounded and Lipschitz. Note that this estimate of the error term $R(\phi, \omega)$ improves on \cite[Eq.\ (63)]{borza2024measure} and we will need this better one (at \eqref{eq:finalkick}, to be precise).
\end{remark}

Finally, \cite{borza2024measure} provides several sufficient and necessary conditions relating the reduced Jacobian and the $\mathsf{MCP}(0, N)$ condition.

\begin{prop}[{\cite[Prop.\ 3.20]{borza2024measure}}]\label{prop:phoenix}
    Let $\hei$ be the sub-Finsler Heisenberg group, equipped with a $C^1$ and strictly convex norm $\normdot$, and with the Lebesgue measure $\Leb^3$. Then, the metric measure space $(\hei, \di, \Leb^3)$ satisfies $\mathsf{MCP}(0, N)$ if and only if 
    \begin{equation}
    \label{eq:equivalent_reducedJ2}
        |\J_R (\phi, \omega t)|\geq t^{N-1} |\J_R (\phi, \omega)|,
    \end{equation}
    for every $(\phi,\omega)\in \mathcal{U}$ and every $t\in [0,1]$.
\end{prop}

 Moreover, when a reference norm $\normdot$ is $C^2$ and strongly convex,
 there is another characterization by using the $\log$-derivative of the reduced Jacobian.

 \begin{prop}[{\cite[Cor.\ 3.22]{borza2024measure}}]
     \label{prop:diffcharacMCP}
     Let $\hei$ be the sub-Finsler Heisenberg group, equipped with a $C^2$ and strongly convex norm $\normdot$, and with the Lebesgue measure $\Leb^3$. Then, the metric measure space $(\hei, \di, \Leb^3)$ satisfies $\mathsf{MCP}(0, N)$ if and only if,
     for all $(\phi,\omega)\in \mathcal{U}$,
     \begin{equation}
         \label{eq:nec&sufMCPdJacineq}
         N(\phi, \omega) := 1 + \frac{\omega \partial_\omega \mathcal{J}_R(\phi, \omega)}{\mathcal{J}_R(\phi, \omega)} \leq N.
     \end{equation}
    
 \end{prop}

\section{Rigidity of the curvature exponent}
\label{sec:rigidity}

In this section, we are going to prove Theorem \ref{thm:rigidity}.
We begin by characterizing the angle correspondence $C_\circ$ of a sub-Riemannian Heisenberg group.

\begin{prop}\label{prop:scalar}
    The angle correspondence $C_\circ$ of a sub-Finsler Heisenberg group $(\hei, \di)$ is (a single-valued) affine map if and only if its norm $\normdot$ is induced by a scalar product, i.e. $(\hei, \di)$ is actually sub-Riemannian.
\end{prop}

\begin{proof} 
Firstly, suppose that a norm $\normdot$ is induced by a scalar product.
Let $\{(a,0),(b_1,b_2)\}$ be an orthonormal basis of the dual norm $\normdot_\ast$,
which is also induced by a scalar product.
Without loss of generality,
we can assume that $ab_2>0$.
Define a curve $t \mapsto(x(t),y(t))$ by
\begin{equation*}
x(t)=a\cos(ct)+b_1\sin(ct),~~~~y(t)=b_2\sin(ct),
\end{equation*}
where $c:=1/(a b_2)$.
Then the generalized trigonometric function $(\cosomp(t),\sinomp(t))$ associated to $\normdot_\ast$ coincides with $(x(t),y(t))$.
Indeed,
clearly the image of $(x(t),y(t))$ is the unit sphere $\{\norm{v}_*=1\}$ and the identity $x\dot{y}-y\dot{x}\equiv 1$ holds by the choice of $c$.
By the formula $\cosomp^{\prime\prime}(t)=-C_\circ^\prime(t) \cosomp(t)$,
we have that $C_\circ^\prime\equiv c^2$.

Conversely,
suppose that $C_\circ^\prime\equiv C^2>0$.
    Recall that the trigonometric functions $(x(t),y(t))=(\cosomp(t),\sinomp(t))$ are recovered by the differential equation
    \begin{equation}
    \label{eq:trig_ode}
        \ddot{x}(t)=-C^2\cdot x(t),\qquad\ddot{y}(t) = -C^2\cdot y(t),
    \end{equation}
    with initial conditions given by $(x(0),y(0))=(x_0,0)$ and $(\dot x(0),\dot y(0))=(\dot x_0,\dot y_0)$. Then, the solution to \eqref{eq:trig_ode} is given by
    \begin{equation}\label{eq:normaltrigonometry}        x(t)=A\cos(Ct)+B_1\sin(Ct),\qquad y(t)=B_2\sin(Ct),
    \end{equation}
    where the constants $A,B_1,B_2\in\R$ are determined by the initial values $x_0$, $\dot x_0$ and $\dot y_0$,
    and we can assume $C>0$ by the symmetries of the (classical) trigonometric functions. Furthermore, the Pythagorean identity \eqref{eq:pytagorean}, Proposition \ref{prop:difftrig} and \eqref{eq:normaltrigonometry} imply the following identity
\begin{align*}
    1&\equiv \cosom(t_\circ) \cosomp(t) + \sinom(t_\circ)\sinomp(t)=\dot{y}(t)x(t)-\dot{x}(t)y(t)\\
&=B_2C\cos(Ct)\left[A\cos(Ct)+B_1\sin(Ct)\right]-\left[-AC\sin(Ct)+B_1C\cos(Ct)\right]B_2\sin(Ct)\\
&=AB_2C,
\end{align*}
showing that $C = 1/(A B_2)$ and that the curve \eqref{eq:normaltrigonometry} is an ellipse. This implies that the norm is induced by a scalar product whose orthonormal basis is $\{(A,0),(B_1,B_2)\}$.
\end{proof}

We now introduce the following finite difference operators, that will be useful in the proof of Theorem \ref{thm:rigidity}.

\begin{definition}
    For $\phi, \omega \in \R$, we define  the first-order finite differences
    \[
    \Delta_\omega C_\circ(\phi) := C_\circ(\phi + \omega) - C_\circ(\phi),\qquad \text{and}\qquad 
    D_\omega C_\circ(\phi) := \frac{\Delta_\omega C_\circ(\phi)}{\omega} = \frac{C_\circ(\phi + \omega) - C_\circ(\phi)}{\omega},
    \]
     as well as the second-order finite differences
    \[
    \Delta_\omega^2 C_\circ(\phi) = \Delta_\omega C_\circ(\phi + \omega) - \Delta_\omega C_\circ(\phi) = C_\circ(\phi + 2 \omega) - 2 C_\circ(\phi + \omega) + C_\circ(\phi),
    \]
    and
    \begin{align*}
        D_\omega^2 C_\circ(\phi) ={}& \frac{\Delta_\omega [D_\omega C_\circ(\phi)]}{\omega} = \frac{D_\omega C_\circ(\phi + \omega) - D_\omega C_\circ(\phi)}{\omega} = \frac{\Delta_\omega C_\circ (\phi + \omega) - \Delta_\omega C_\circ(\phi)}{\omega^2} \\
        ={}& \frac{\Delta_\omega^2 C_\circ(\phi)}{\omega^2} =  \frac{C_\circ(\phi + 2 \omega) - 2 C_\circ(\phi + \omega) + C_\circ(\phi)}{\omega^2}.
    \end{align*}
    
\end{definition}

We are going to prove a series of lemmas before concluding with the proof of Theorem \ref{thm:rigidity}. In these lemmas, $(\hei, \di, \Leb^3)$ is a sub-Finsler Heisenberg group associated with a $C^1$ and strongly convex norm that is not induced by a scalar product. In particular, the angle correspondence map $C_\circ$ is strictly increasing and Lipschitz, thus $\Lip(C_\circ) < +  \infty$ (see \cite[Prop.\ 2.16]{borza2024measure}), and is not affine by Proposition \ref{prop:scalar}.  We define $a \geq 0$ as the quantity given by
\begin{equation}\label{eq:assumptiona>0}
        a:= \essinf_{\phi \in [0,2\pi_{\Omega^\circ})} C'_\circ (\phi) \leq \Lip(C_\circ) < +\infty.
\end{equation}
In this definition and in the rest of the paper, whenever we consider an interval of generalized angles $I\subset\R$, this is identified, with a slight abuse of notation, with its quotient in $\R/2\pi_{\Omega^\circ}\mathbb Z$.

\begin{lemma}
\label{lemma:D2>H}
        There exists a positive constant $H$ such that for every $h>0$, there are $\delta\in (0,h]$ and $\phi \in \rinterval{0}{2 \pi_{\Omega^\circ}}$ for which it holds that 
    \begin{equation}\label{eq:estimatedelta}
        D^2_\delta C_\circ (\phi) \geq H.
    \end{equation}
    \end{lemma}

\begin{proof}
    Observe that \eqref{eq:periodicity_Ccirc} implies that the function
    \begin{equation}
       [0,2\pi_{\Omega^\circ})\ni \phi \mapsto C_\circ(\phi) - \frac{\pi_\Omega}{\pi_{\Omega^\circ}} \phi
    \end{equation}
    is $\pi_{\Omega^\circ}$-periodic and non-constant. Thus it is not concave on $[0,2\pi_{\Omega^\circ})$. Consequently, $C_\circ$ is not concave and we can find $\phi_0$ and $\omega_0$ such that $D_{\omega_0}^2 C_\circ(\phi_0)=:H>0$. To prove this, we argue by contradiction, assuming that $D_{\omega_0}^2 C_\circ(\phi_0)\leq 0$, for every choice of angles $\phi_0$, $\omega_0$. But this implies that $C_\circ$ is midpoint concave, and thus concave. 
    
    Now, assume by contradiction that there exists $h>0$ such that, for every $\delta\in (0,h]$ and every $\phi \in \rinterval{0}{2 \pi_{\Omega^\circ}}$, we have $D^2_\delta C_\circ (\phi) < H$. From the very definition of $D_\omega^2C_\circ(\phi)$, we have that
    \begin{equation}
        D_{\omega_0}^2C_\circ(\phi_0)=\frac{1}{4}D^2_{\omega_0/2}C_\circ(\phi_0)+\frac{1}{2}D_{\omega_0/2}^{2}C_\circ(\phi_0+\omega_0/2)+\frac{1}{4}D_{\omega_0/2}^2C_\circ(\phi_0+\omega_0).
    \end{equation}
    Iterating this identity, for every $n$ we obtain
    \begin{align*}
        D_{\omega_0}^2C_\circ(\phi_0) = \frac{1}{2^{2n}}{}& \sum_{k=0}^{2^n-1} (k+1) D^2_{\omega_0/2^n} C_\circ(\phi_0 + k\omega_0/2^n) \\
        &+ \frac{1}{2^{2n}}\sum_{k=2^n}^{2^{n+1}-2} (2^{n+1}-1-k) D^2_{\omega_0/2^n}C_\circ (\phi_0 + k\omega_0/2^n).
    \end{align*}
   Then, taking $n$ sufficiently large such that $\omega_0/2^n<h$, we deduce that 
   \begin{equation}
       D_{\omega_0}^2C_\circ(\phi_0) <  \frac H{2^{2n}} \Bigg[\sum_{k=0}^{2^n-1} (k+1) + \sum_{k=2^n}^{2^{n+1}-2} (2^{n+1}-1-k) \Bigg]= \frac H{2^{2n}} \Bigg[ 2^n + 2 \sum_{k=1}^{2^{n}-1} k \Bigg] = H,
   \end{equation}
   finding a contradiction and concluding the proof.
\end{proof}

Note that with the assumptions and notations of Lemma \ref{lemma:D2>H}, we also have that
    \begin{equation}
    \label{eq:TaylorCcirc}
        C_\circ(\phi + 2 \delta) \geq C_\circ(\phi) + 2 \delta D_\delta C_\circ (\phi) + H \delta^2,
    \end{equation}
in fact the following identity holds:
    \begin{align*}
        \Delta_{2 \delta} C_\circ(\phi) ={}& 2 (C_\circ(\phi + \delta) -  C_\circ(\phi)) + C_\circ(\phi + 2 \delta) - 2 C_\circ(\phi + \delta) + C_\circ(\phi) \\
        ={}& 2 \delta D_\delta C_\circ(\phi) + D^2_\delta C_\circ(\phi) \delta^2.
    \end{align*}
In Lemma \ref{lemma:Deltat<=} and Lemma \ref{lemma:<=Deltat<=}, $H$ will stand for the quantity given by Lemma \ref{lemma:D2>H}.

\begin{lemma}
\label{lemma:Deltat<=}
    For every $h > 0$, there are $\delta \in \linterval{0}{h}$, $\phi \in \rinterval{0}{2 \pi_{\Omega^\circ}}$, $\omega \in \linterval{0}{2 \delta}$, and
    \begin{equation}
        \label{eq:K>=A}
        K \geq A := \essinf_{t \in \rinterval{\phi}{\phi + \omega}} C_\circ'(t) \geq a
    \end{equation}
    for which it holds that
    \begin{equation}
        \label{eq:Deltat<=}
        \Delta_t C_\circ(\phi)\leq \left(K + \frac{H \delta}{2}\right) t,\quad\forall t \in \interval{0}{\omega} \qquad \text{and}\qquad \Delta_\omega C_\circ(\phi)=\left(K + \frac{H \delta}{2}\right) \omega.
    \end{equation}
\end{lemma}

\begin{proof}
    Given any $h>0$, we apply Lemma \ref{lemma:D2>H} finding $\delta\in (0,h]$ and $\overline{\phi} \in \rinterval{0}{2 \pi_{\Omega^\circ}}$ for which
    \[
        D^2_\delta C_\circ (\overline{\phi}) \geq H.
    \]
     We set $K := D_\delta C_\circ(\overline{\phi})$ and define
    \begin{align*}
        \phi :={}& \sup \left\{ \theta \in \interval{\overline{\phi}}{\overline{\phi} + \delta} : C_\circ(\theta) - C_\circ(\overline{\phi}) \geq \left(K + \frac{H \delta}{2}\right) (\theta - \overline{\phi}) \right\}, \\
        \omega :={}& \inf \left\{ t \in \linterval{0}{\overline{\phi} + 2 \delta - \phi} : \Delta_tC_\circ(\phi) \geq \left(K + \frac{H \delta}{2}\right) t \right\}.
    \end{align*}
    Observe that, by definition of $K$,
    \[
    C_\circ(\overline{\phi} + \delta) - C_\circ(\overline{\phi}) = K  \delta < \left(K + \frac{H \delta}{2}\right)\delta,
    \]
    which ensures that $\phi$ is well defined and that $\phi < \overline{\phi} + \delta$. By continuity of $C_\circ$, it actually holds that
    \[
    C_\circ(\phi) - C_\circ(\overline{\phi}) = \left(K + \frac{H \delta}{2}\right) (\phi - \overline{\phi}).
    \]
    Moreover, taking into account \eqref{eq:TaylorCcirc}, we note that
    \begin{align*}
        C_\circ(\phi + (\overline{\phi} + 2 \delta - \phi)) - C_\circ(\phi) ={}& C_\circ(\overline{\phi} + 2 \delta) - C_\circ(\overline{\phi}) + C_\circ(\overline{\phi}) - C_\circ(\phi) \\
        \geq{}& 2 \delta K + H \delta^2 + \left(K + \frac{H \delta}{2}\right) (\overline{\phi} - \phi) \\
        ={}& \left(K + \frac{H \delta}{2}\right) (\overline{\phi} + 2 \delta - \phi).
    \end{align*}
    Thus, the set that determines $\omega$ is not empty, which implies that $\omega$ is well-defined. Furthermore, the very definition of $\phi$ guarantees that $\omega > \overline\phi + \delta - \phi > 0$. Then, since $C_\circ$ is continuous, \eqref{eq:Deltat<=} follows from the definitions of $\phi$ and $\omega$. Moreover, we observe that 
    \begin{equation}
    \begin{split}
         C_\circ(\overline\phi + \delta) - C_\circ(\phi)=  C_\circ(\overline\phi + \delta)- &C_\circ(\overline \phi) - \big[C_\circ(\phi)- C_\circ( \overline\phi)\big] \\
         &= K \delta  - \left(K + \frac{H \delta}{2}\right) (\phi - \overline{\phi}) \leq K (\overline\phi + \delta - \phi)
    \end{split}
    \end{equation}
    and therefore we deduce that
    \begin{equation}
        K \geq \essinf_{t \in \rinterval{\phi}{\overline \phi + \delta}} C_\circ'(t) \geq A \geq a,
    \end{equation}
    proving \eqref{eq:K>=A}.
\end{proof}

Given $C > 0$, 
we define
$S_C := \left\{ \psi \in \rinterval{0}{2 \pi_{\Omega^\circ}} : C_\circ'(\psi) \leq C \right\}$. Moreover, for a subset $S\subset \R$, we denote by $D(S)$ the set of density points of $S$, which are contained in $S$.

\begin{lemma}
\label{lemma:<=Deltat<=}
    For every $h > 0$, there are $\delta \in \linterval{0}{h}$ and $A \geq a$ that satisfy the following: for all $\epsilon > 0$ sufficiently small, there are $\phi \in D(S_{A + \epsilon})$, $\omega \in \linterval{0}{2 \delta}$, and $B \geq H \delta/2$ for which it holds
    \begin{equation}
        \label{eq:<=Deltat<=}
         A t \leq \Delta_t C_\circ(\phi) \leq (A + B) t \quad\text{for all } t \in \interval{0}{\omega}\qquad \text{and}\qquad  \Delta_\omega C_\circ(\phi) = (A + B) \omega.
    \end{equation}
\end{lemma}

\begin{proof}
     Given any $h>0$, we apply Lemma \ref{lemma:Deltat<=} finding $\delta \in \linterval{0}{h}$, $\overline\phi \in \rinterval{0}{2 \pi_{\Omega^\circ}}$, $\overline\omega \in \linterval{0}{2 \delta}$, and $K \geq A := \essinf_{t \in \rinterval{\overline\phi}{\overline\phi + \overline\omega}} C_\circ'(t) \geq a$ for which
      \begin{equation}\label{eq:ddddd}
         \Delta_t C_\circ(\overline\phi)\leq \left(K + \frac{H \delta}{2}\right) t,\quad\forall t \in \interval{0}{\overline\omega} \qquad \text{and}\qquad \Delta_{\overline\omega} C_\circ(\overline\phi)=\left(K + \frac{H \delta}{2}\right) \overline\omega.
    \end{equation}
    Let $\epsilon > 0$ sufficiently small and pick any $\phi \in D(S_{A + \epsilon}(\overline{\phi}, \overline{\omega}))$, where we denote
    \[
    S_{A + \epsilon}(\overline{\phi}, \overline{\omega}) := \left\{ \theta \in \rinterval{\overline{\phi}}{\overline{\phi} + \overline{\omega}} : C_\circ'(\theta) \leq A + \epsilon \right\} \subseteq S_{A + \epsilon}.
    \]
    Note that, by definition of $A$, this set is not empty and in addition $\varphi=\overline\varphi+\overline t$, for some $\overline t\in [0,\bar\omega]$. Now, we introduce
    \begin{align*}
        B :={}& \frac{C_\circ(\overline{\phi} + \overline{\omega}) - C_\circ(\phi)}{\overline{\phi} + \overline{\omega} - \phi} - A = \frac{C_\circ(\overline{\phi} + \overline{\omega}) - C_\circ(\overline{\phi}) + C_\circ(\overline{\phi}) - C_\circ(\phi)}{\overline{\phi} + \overline{\omega} - \phi} - A \\
        ={}&\frac{\Delta_{\overline\omega}C_\circ(\overline\varphi) -\Delta_{\overline t} C_\circ(\overline\phi)}{\overline{\phi} + \overline{\omega} - \phi} -A
        \geq \frac{(K + H \delta/2) (\overline{\phi} + \overline{\omega} - \phi)}{\overline{\phi} + \overline{\omega} - \phi} - A \geq \frac{H \delta}{2} > 0,
    \end{align*}
    where the first inequality follows from \eqref{eq:ddddd}. Now, we take
    \[
    \omega := \inf \{t \in \linterval{0}{\overline{\phi} + \overline{\omega} - \phi} : \Delta_tC_\circ(\phi) = C_\circ(\phi + t) - C_\circ(\phi) \geq (A + B) t \}.
    \]
    The definition of $B$ ensures that
    \[
    C_\circ(\phi + (\overline{\phi} + \overline{\omega} - \phi)) - C_\circ(\phi) = C_\circ(\overline{\phi} + \overline{\omega}) - C_\circ(\phi) = (A + B) (\overline{\phi} + \overline{\omega} - \phi)
    \]
    and thus that $\omega$ is well defined. Moreover, since $\phi\in D(S_{A + \epsilon}(\overline{\phi}, \overline{\omega}))\subset S_{A+\varepsilon}$, then $C_\circ'(\varphi)\leq A+\varepsilon$ and we can deduce that $\omega > 0$.  Now, the lower bound of \eqref{eq:<=Deltat<=} is a consequence of the definition of $A$, while the upper bound follows from the definitions of $\phi$ and $\omega$. The equality $\Delta_\omega C_\circ(\phi) = (A + B) \omega$ is justified by the continuity of $C_\circ$.
\end{proof}

For the convenience of the reader, we recall here the definition of $P(\varphi,\omega)$ given in \eqref{eq:JRaround0}:
\begin{equation}
P(\phi, \omega) :=\frac{1}{2} \int_{\phi}^{\phi + \omega} \left( \int_{\phi}^{\phi + \omega} (t - s)^2 C_\circ'(t) C_\circ'(s) \diff s \right) \diff t,\qquad \forall\,(\varphi,\omega)\in \mathcal U.
\end{equation}

\begin{lemma}
\label{lem:integrationbypartsmanymanytimes}
    For all $(\phi, \omega) \in \mathcal{U}$, we have that
    \begin{equation}\label{eq:partsandparts}
        P(\phi, \omega) = 2 \Delta_\omega C_\circ(\phi) \int_0^\omega (\omega - t) \Delta_t C_\circ(\phi) \diff t - \bigg(\int_0^\omega \Delta_t C_\circ(\phi) \diff t\bigg)^2.
    \end{equation}
\end{lemma}

\begin{proof}
Integrating by parts a first time we obtain
    \begin{equation}
    \begin{split}
         P(\phi, \omega)&= \frac 12\int_{0}^{ \omega} (t - \omega)^2 C_\circ'(\phi + t) C_\circ(\phi + \omega) \diff t - \frac 12\int_{0}^{ \omega} t^2 C_\circ'(\phi + t) C_\circ(\phi) \diff t \\
        &\quad +\int_{0}^{\omega}  \int_{0}^{ \omega} (t - s) C_\circ'(\phi +t) C_\circ(\phi +s) \diff s \diff t.
    \end{split}
    \end{equation}
With another integration by parts and after standard computations, we have that:
    \begin{equation}
    \begin{split}
        P(\phi, \omega)&= - \omega^2  C_\circ(\phi +\omega) C_\circ(\phi) + 2 \int_{0}^{ \omega} (\omega - t) C_\circ(\phi +\omega) C_\circ(\phi +t) \diff t \\
        &\quad + 2 \int_{0}^{ \omega} t \,  C_\circ(\phi) C_\circ(\phi +t) \diff t -\int_{0}^{\omega}  \int_{0}^{ \omega} C_\circ(\phi +t) C_\circ(\phi +s) \diff s \diff t.
    \end{split}
    \end{equation}
We add and remove the quantity $\omega^2C_\circ(\phi)^2+ 2 \omega C_\circ(\phi) \int_{0}^{ \omega}  C_\circ(\phi +t) \diff t$, obtaining:
\begin{equation}
    \begin{split}
        P(\phi, \omega) ={}& - \omega^2  C_\circ(\phi)^2 - \omega^2 C_\circ(\phi) [C_\circ(\phi +\omega)- C_\circ(\phi)]   + 2 \omega C_\circ(\phi) \int_{0}^{ \omega}  C_\circ(\phi +t) \diff t \\
        &\quad + 2 [C_\circ(\phi +\omega)- C_\circ(\phi)] \int_{0}^{ \omega} (\omega - t)  C_\circ(\phi +t) \diff t -  \bigg(\int_{0}^{ \omega} C_\circ(\phi +t) \diff t \bigg)^2\\
        ={}&[C_\circ(\phi +\omega)-C_\circ(\phi)] \bigg( \int_{0}^{ \omega} (\omega - t) C_\circ(\phi +t) \diff t  - \omega^2 C_\circ(\phi)  \bigg) \\
        &\quad -\bigg(\int_{0}^{ \omega} C_\circ(\phi +t) \diff t- \omega C_\circ(\phi) \bigg)^2
    \end{split}
\end{equation}
Once we have this expression, the thesis simply follows from the definition of $\Delta_\omega C_\circ(\phi)$.
\end{proof}

We are now ready to prove our first main result.

\begin{theorem}
    \label{thm:rigidityinsection}
    Any Heisenberg group $\hei$, equipped with a sub-Finsler norm $\normdot$ and the Lebesgue measure $\Leb^3$, satisfies $N_{\mathrm{curv}} \geq 5$. Furthermore, $N_{\mathrm{curv}} = 5$ if and only if $\normdot$ is induced by a scalar product.
\end{theorem}

\begin{proof}
  If $\normdot$ is induced by a scalar product, then there is an isometry between $(\hei, \di)$ and the standard sub-Riemannian Heisenberg group, and thus $(\hei, \di,\Leb^3)$ has $N_{\text{curv}}= 5$ (see \cite{juillet2009}). 
  
  If $\normdot$ is not induced by a scalar product and it is neither strongly convex nor $C^1$, then, according to \cite[Theorem 1.1]{borza2024measure}, it does not satisfy $\mathsf{MCP}(K, N)$ for any $K \in \R$ and any $N \in \ointerval{1}{+\infty}$ and thus $N_{\mathrm{curv}} = +\infty$. For the rest of the proof, we assume that $\normdot$ is $C^1$ and strongly convex and it is not induced by a scalar product, i.e. $C_\circ$ is not an affine map, by Proposition \ref{prop:scalar}. We are going to prove that $N_{\mathrm{curv}} > 5$.
    
    We place ourselves within the conclusions of Lemma \ref{lemma:<=Deltat<=}, using the same notations except for $S_{A + \epsilon}$ that we denote by $S$ for simplicity. Recall the definition of $a$ in \eqref{eq:assumptiona>0} and consider $r \in \interval{0}{1}$.
    
    
    \paragraph{Step 1: Finding an upper bound on $P(\phi, r \omega)$.\\} 
    
    Let $\phi\in D(S_{A+\varepsilon})$ the angle identified by Lemma \ref{lemma:<=Deltat<=}.
    By writing $\interval{\varphi}{\varphi+r \omega}$ as the disjoint union of $\interval{\varphi}{\varphi+r \omega} \cap S$ (on which we have $C_\circ'(\theta) \leq A + \epsilon$) and $\interval{\varphi}{\varphi+r \omega} \setminus S$ (on which we use the trivial estimate $C_\circ'(\theta) \leq \Lip(C_\circ)$), using the definition of $P$ in \eqref{eq:JRaround0}, we obtain
   
    \begin{align*}
        P(\phi, r\omega) 
        \leq \frac{1}{2} \Big({}& (A + \epsilon)^2 
        \int_{\interval{\varphi}{\varphi+r \omega}\cap S}  \int_{\interval{\varphi}{\varphi+r \omega}\cap S} (t - s)^2 \diff s \diff t \\
        & + 2 (A + \epsilon) \Lip(C_\circ) \int_{\interval{\varphi}{\varphi+r \omega}\setminus S}  \int_{\interval{\varphi}{\varphi+r \omega}\cap S} (t - s)^2 \diff s \diff t \\
        & + \Lip(C_\circ)^2\int_{\interval{\varphi}{\varphi+r \omega}\setminus S} \int_{\interval{\varphi}{\varphi+r \omega}\setminus S} (t - s)^2 \diff s \diff t \Big).
    \end{align*}
    We then use the trivial inclusion $\interval{\varphi}{\varphi+r \omega} \cap S \subseteq \interval{\varphi}{\varphi+r \omega}$ to get
    \begin{align*}
        P(\phi, r\omega) 
        \leq \frac{1}{2} \Big({}
        & (A + \epsilon)^2 
        \int_{\interval{\varphi}{\varphi+r \omega}}  \int_{\interval{\varphi}{\varphi+r \omega}} (t - s)^2 \diff s \diff t \\
        & + 2 (A + \epsilon) \Lip(C_\circ) \int_{\interval{\varphi}{\varphi+r \omega}\setminus S}  \int_{\interval{\varphi}{\varphi+r \omega}} (t - s)^2 \diff s \diff t \\
        & + \Lip(C_\circ)^2 \int_{\interval{\varphi}{\varphi+r \omega}\setminus S} \int_{\interval{\varphi}{\varphi+r \omega}\setminus S} (t - s)^2 \diff s \diff t \Big) \\
        \leq \frac{1}{2} \Big({}& \frac{1}{6} (A + \epsilon)^2 r^4 \omega^4 + \frac{4}{3} (A + \epsilon) \Lip(C_\circ) r^3 \omega^3 \Leb^1(\interval{\phi}{\phi + r \omega}\setminus S) \\
        &+  \Lip(C_\circ)^2 r^2 \omega^2 \Leb^1(\interval{\phi}{\phi + r \omega}\setminus S)^2 \Big).
    \end{align*}
    Since $\phi$ is a density point of $S$, we know that
    \[
    \frac{\Leb^1(\interval{\phi}{\phi + r \omega}\setminus S)}{r \omega} \to 0, \qquad\text{as } r \to 0^+.
    \]
    In particular, we deduce that for every $r > 0$ small enough, it holds
    \begin{equation}
        \label{eq:Pupperbound}
        P(\phi, r \omega) \leq \frac{1}{12} (A + 2 \epsilon)^2 r^4 \omega^4.
    \end{equation}
    

    \paragraph{Step 2: Finding a lower bound on $P(\phi, \omega)$.\\} Considering the function $f_\phi(t) := \Delta_t C_\circ(\phi) - A t$, from \eqref{eq:<=Deltat<=}, we deduce that
    \begin{equation}
        \label{eq:<=Gt<=}
        0 \leq f_\phi(t) \leq B t \ \ \text{ for all } t \in \interval{0}{\omega}, \text{ and } \ \ f_\phi(\omega) = B \omega.
    \end{equation}
    To deduce the lower bound on $P$, we use the expression \eqref{eq:partsandparts} obtained in Lemma \ref{lem:integrationbypartsmanymanytimes}. Then, thanks to our choice of angles, we have
    \begin{align}
        \label{eq:PwithG}
        \begin{split}
            P(\phi, \omega) ={}& 2 \omega (A + B) \left(\int_0^\omega (\omega - t) f_\phi(t) \diff t + \frac{1}{6} A \omega^3 \right) - \left( \frac{1}{2} A \omega^2 + \int_0^\omega f_\phi(t) \right)^2 \\
            ={}& \frac{1}{12} A^2 \omega^4 + \frac{1}{3} A B \omega^4 + A \omega \left( 2 \int_0^{\omega} (\omega - t) f_\phi(t) \diff t - \omega \int_0^{\omega}  f_\phi(t) \diff t \right) \\
            & + \left( 2 B \omega \int_0^{\omega} (\omega - t) f_\phi(t) \diff t - \left(\int_0^{\omega}  f_\phi(t) \diff t\right)^2\right).
        \end{split}
    \end{align}
    We now estimate the two quantities within the parentheses. Firstly, we observe that 
    \begin{equation*}
        2 \int_0^{\omega} (\omega - t) f_\phi(t) \diff t - \omega \int_0^{\omega}  f_\phi(t) \diff t = \int_0^{\omega} (\omega -2 t) f_\phi(t) \diff t \geq \int_{\omega/2}^{\omega} (\omega - 2 t) Bt \diff t = -\frac{5}{24} B \omega^3.
    \end{equation*}
    Secondly, we claim that the third line of \eqref{eq:PwithG} is non-negative. To prove this, we note that \eqref{eq:<=Gt<=} implies the existence of a unique $s \in \interval{0}{\omega}$ satisfying $\int_{s}^{\omega} B t \diff t = \int_{0}^\omega f_\phi(t) \diff t$. Then, we have that
    \[
    \int_0^\omega (f_{\phi}(t) - \chi_{\interval{s}{\omega}} Bt) \diff t = 0,\quad \text{and thus}\quad \int_0^s f_\phi(t) \diff t = \int_s^\omega (Bt - f_{\phi(t)}) \diff t.
    \]
    This yields that
    \begin{equation}
        \int_0^s (\omega - t) f_\phi(t) \diff t \geq \int_0^s (\omega - s) f_\phi(t) \diff t = \int_s^\omega (\omega - s) (B t - f_\phi(t)) \diff t \geq \int_s^\omega (\omega - t) (B t - f_\phi(t)) \diff t,
    \end{equation}
    and therefore $\int_0^{\omega} (\omega - t) f_\phi(t) \diff t \geq \int_s^{\omega} (\omega - t) B t \diff t$. Hence, we can estimate the third line of \eqref{eq:PwithG} as follows:
    \begin{multline*}
        2 B \omega \int_0^{\omega} (\omega - t) f_\phi(t) \diff t - \left(\int_0^{\omega}  f_\phi(t) \diff t\right)^2\\ 
        \geq 2 B \omega \int_s^{\omega} (\omega - t) B t \diff t - \left(\int_s^{\omega}  B t \diff t\right)^2 = \frac{1}{12} B^2 (\omega - s)^3 (3 s + w) \geq 0. 
    \end{multline*}
    Finally, we conclude that
    \begin{equation}
        \label{eq:Plowerbound}
        P(\phi, \omega) \geq \frac{1}{12} A^2 \omega^4 + \frac{1}{8} A B \omega^4.
    \end{equation}


    \paragraph{Step 3: Proving $N_{\mathrm{curv}} > 5$ when $a > 0$.\\} Using the estimates \eqref{eq:Pupperbound} and \eqref{eq:Plowerbound}, keeping in mind \eqref{eq:boundremainder}, we deduce that for $r$ sufficiently small, we have
    \begin{align*}
        \frac{\mathcal{J}_R(\phi, r \omega)}{\mathcal{J}_R(\phi, \omega)} \leq{}& \frac{\frac{1}{12} (A + 2 \epsilon)^2 r^4 \omega^4 + M r^6 \omega^6}{\frac{1}{12} A^2 \omega^4 + \frac{1}{8} A B \omega^4 - M \omega^6} = r^4 \frac{\frac{1}{12} (A + 2 \epsilon)^2 + M r^2 \omega^2}{\frac{1}{12} A^2 + \frac{1}{8} A B  - M \omega^2}.
    \end{align*}
    Moreover, we observe that 
    \begin{equation}
    \label{eq:explimit}
        \lim_{r\to 0} \frac{\frac{1}{12} (A + 2 \epsilon)^2 + M r^2 \omega^2}{\frac{1}{12} A^2 + \frac{1}{8} A B  - M \omega^2} = \frac{(A+2 \varepsilon)^2}{A^2 + \frac 32 AB - 12 M \omega^2} \leq \frac{(A+2 \varepsilon)^2}{A^2 + \frac 34 H A \delta - 48 M  \delta^2},
    \end{equation}
    where the inequality follows from the facts that $B \geq H \delta/2$ and $\omega\leq 2 \delta$. We finally choose $h > 0$ and $\epsilon > 0$ in Lemma \ref{lemma:<=Deltat<=} so that the expression in \eqref{eq:explimit} is smaller than 1. For example, we can take
    \[
    h \leq \frac{a H}{192 M}\qquad \text{and}\qquad \epsilon \leq \frac{1}{2} \left( \sqrt{A^2 + \frac 14 H a \delta} - A \right).
    \]
    Note that, since $a > 0$, we may pick $h, \epsilon > 0$. With these choices, recalling that $A\geq a$ and $\delta\leq h$, we get
    \begin{equation}
        \label{eq:finalkick}
        A^2 + \frac 34 H A \delta - 48 M  \delta^2 \geq A^2 + \frac 34 H a \delta - 48 M  \delta^2 \geq A^2 + \frac 12 H a \delta > (A+2 \varepsilon)^2.
    \end{equation}

   Putting everything together, we deduce that there exists $r\in \interval{0}{1}$ (close to $0$) such that 
    \begin{equation}
        \frac{\J_R (\phi, r \omega)}{\J_R (\phi, \omega)} < r^4.
    \end{equation}
    According to Proposition \ref{prop:phoenix}, this is sufficient to prove that the metric measure space $(\hei, \di, \Leb^3)$ does not satisfy $\mathsf{MCP}(0,5)$, and then $N_{\text{curv}}> 5$, thanks to Remark \ref{lem:lsc_curv_exp}.


    \paragraph{Step 4: Proving $N_{\mathrm{curv}} > 5$ when $a = 0$.\\} In this case, we can provide a similar and simpler argument. Since $a = 0$, the set $S_\epsilon$ is non-empty for all $\epsilon > 0$. Therefore, repeating the argument from Step 1 with $A = 0$, we obtain that if $\phi \in D(S_\epsilon)$ and $\omega > 0$, then for all $r \in \interval{0}{1}$ sufficiently small
    \[
    P(\phi, r \omega) \leq 3 \epsilon^2 r^4 \omega^4. 
    \]
     Recalling Proposition \ref{prop:formulaJRdJRint}, this yields that
    \begin{equation}
        \label{eq:failMCPa=0}
        \lim_{r \to 0} \frac{\J_R(\phi, r \omega)}{r^4} \leq \lim_{r \to 0} \frac{1}{r^4} \left( 3 \epsilon^2 r^4 \omega^4 + M r^2 \omega^5 \right) = 3 \epsilon^2 \omega^4.
    \end{equation}
    As we are focusing on $C^1$ and strongly convex reference norms, the reduced Jacobian $(\phi,\omega) \mapsto \J_R (\phi,\omega)$ is continuous and nonnegative when $\omega$ is not a multiple of $2\pi_{\Omega^\circ}$, see Remark 3.12 and Proposition 3.16 in \cite{borza2024measure}. In particular, by periodicity, we have that 
    \begin{equation}
        u:=\min_{\theta \in [0,2\pi_{\Omega^\circ})} \J_R (\theta,\pi_{\Omega^\circ}) >0.
    \end{equation}
    Finally, we choose $\epsilon < \sqrt{u/(3 \pi_{\Omega^\circ}^4)}$ and $\omega = \pi_{\Omega^\circ}$ and we get that for $r\in[0,1]$ sufficiently small it holds that
    \begin{equation}
        \J_R (\phi,  r \pi_{\Omega^\circ}) < r^4 u\leq r^4 \J_R (\phi,  \pi_{\Omega^\circ}).
    \end{equation}
    As in the fist case, this is sufficient to conclude.
\end{proof}

\begin{remark}
    We would expect that Theorem \ref{thm:rigidityinsection} could be extended to higher-dimensional sub-Finsler Heisenberg groups $\hei_n$ with rigidity at $N_{\mathrm{curv}} = 2 n + 3$. However, this seems difficult with the techniques developed here. The main reason is that, in general, we do not have an explicit expression of \sF geodesics for higher-dimensional Heisenberg groups. The expression of geodesics is available solely for a specific class of norms, see \cite[Main Assumption]{lok2}.
\end{remark}

\section{Continuity of the curvature exponent}
\label{sec:continuity}

In this section, we prove that the curvature exponent of \sF Heisenberg groups is continuous with respect to a suitable convergence of the underlying norms. As a consequence, we deduce that for every $N^\star\geq5$ there exists a \sF Heisenberg group having $N^\star$ as curvature exponent. See Section \ref{sec:convex_trigtrig} for the definitions of strongly convex and $C^k$ norms.

\begin{definition}[$C^k$-strong convergence of norms]\label{def:strong_convergence}
    Let $k\geq 2$ and $\{\normdot_\varepsilon\}_{\varepsilon > 0}$ be a family of $C^k$ and strongly convex norms on $\R^2$. We say that $\{\normdot_\varepsilon\}_{\varepsilon> 0 }$ converges $C^k$-strongly to a $C^k$ and strongly convex norm $\normdot_{0}$ on $\R^2$ as $\varepsilon \to 0$ if $\normdot^2_\varepsilon\to\normdot^2_{0}$ in $C^k_{loc}(\R^2\setminus\{0\})$ as $\varepsilon \to 0$.
\end{definition}

\begin{remark}
\label{rmk:dual_is_bello}
     If the family $\{\normdot_\varepsilon\}_{\varepsilon > 0}$ of $C^k$ and strongly convex norms on $\R^2$ converges $C^{k}$-strongly to a $C^k$ and strongly convex norm $\normdot_{0}$ as $\varepsilon \to 0$, then $(\normdot_{\varepsilon})_*^2\to(\normdot_0)_*^2$ in $C^k_{loc}(\R^2\setminus\{0\})$ as $\varepsilon \to 0$. Indeed, for a given $C^k$ and strongly convex norm $\normdot$, according to \cite[Lem.\ 2.9]{magnabosco2023failure}, it holds that 
     \begin{equation}
         v^*=\norm{v} d_v \normdot = \frac12 d_v \normdot^2:= N(v), \qquad\forall\, v\in \R^2\setminus\{0\}, 
     \end{equation}
     where $v^*$ denotes the dual element of the vector $v$. Thus, defining $N^*(\lambda):=\frac12d_{\lambda} \normdot_*^2$, as $v^{**}=v$, we can deduce the identities
     \begin{equation}
     \label{eq:easy_id}
         N\circ N^*=N^*\circ N=\id_{\R^2\setminus\{0\}}.
     \end{equation}
     If $\normdot_\varepsilon\to\normdot_0$ $C^k$-strongly, then $N_\varepsilon$ converges to $N_0$ in the $C^{k-1}$-topology $\varepsilon \to 0$. From \eqref{eq:easy_id}, we see that also $N_\varepsilon^*$ converges to $N_0^*$ in the $C^{k-1}$-topology as $\varepsilon \to 0$, concluding the proof of the claim.
\end{remark}

\begin{lemma}
\label{lem:regularity_P}
    Let $k\geq 2$ and $\normdot$ be a $C^k$ and strongly convex norm. The map $P:\R/ 2 \pi_\Omega \mathbb Z \to  \partial \Omega \subset \R^2$ that associates to the angle $\theta$ the vector $P_\theta$, is a $C^k$-diffeomorphism.
\end{lemma}

\begin{proof}
    Recall that the map $P$ is bi-Lipschitz, cf.\ \cite[Lem.\ 2.8]{borza2024measure}. Moreover, if $\normdot$ is of class $C^k$, then $\partial\Omega=\{\norm{v}=1\}$ is a $C^k$ curve, as $1$ is a regular value of $f(v):=\norm{v}^2$. Let $\gamma:[0,1]\to \partial\Omega$ be a $C^k$ parametrization of $\partial\Omega$, with $\gamma(0)=P(0)$ and constant (Euclidean) speed. Then, there exists a $C^k$ function $\theta:[0,1]\to [0,2\pi_\Omega]$ such that
    \begin{equation}
        \gamma(t)=P(\theta(t)),\qquad\forall\, t\in [0,1],
    \end{equation}
    where, by definition, $\theta(0)=0$ and $\theta(1)=2\pi_\Omega$. Indeed, using Stokes' theorem, we see that the function $\theta(\cdot)$ satisfies
    \begin{equation}
    \label{eq:def_theta}
        \theta(t)= \int_0^t (\gamma_1(s)\dot\gamma_2(s)-\gamma_2(s)\dot\gamma_1(s))\de s,
    \end{equation}
    and thus is $C^k$. Now, we compute 
    \begin{equation}
        \dot\theta(t)=\gamma_1(t)\dot\gamma_2(t)-\gamma_2(t)\dot\gamma_1(t)= \gamma(t)\cdot\dot\gamma(t)^\perp,
    \end{equation}
    where $\cdot$ denotes here the Euclidean scalar product and $v^\perp$ is the oriented Euclidean orthogonal complement of $v$. Since $\dot\gamma(t)\neq 0$ for every $t\in [0,1]$, we deduce that also $\dot\theta(t)\neq 0$ for every $t\in [0,1]$. Indeed, we have
    \begin{equation}
        0=\frac12\frac{\de}{\de t}|\dot\gamma(t)|^2= \gamma(t)\cdot\dot\gamma(t),\qquad\forall\,t\in [0,1], 
    \end{equation}
    therefore $\dot\theta(t)= 0$ would imply that $\dot\gamma(t)$ is parallel and orthogonal to $\gamma(t)$, meaning that $\dot\gamma(t)=0$ and thus giving a contradiction. Therefore, $\theta$ is a $C^k$-diffeomorphism and thus $P=\gamma\circ\theta^{-1}$ is $C^k$-diffeomorphism as well.   
\end{proof}

\begin{lemma}
\label{lem:curvescurvescurves}
    Let $f_\varepsilon,f:\R^2\to\R$ be proper $C^k$ functions such that 
    \begin{equation*}
        f_\varepsilon\to f,\qquad \text{in }C^k_{loc}(\R^2,\R). 
    \end{equation*}
    Assume that $1$ is a regular value for $f_\varepsilon,f$ and let $\gamma^\varepsilon,\gamma:[0,1]\to \R^2$ be constant-speed parametrization of $\{f_\varepsilon=1\}$ and $\{f=1\}$, respectively, such that $\gamma^\varepsilon(0) \to\gamma(0)$ and $\gamma^\varepsilon(1) \to\gamma(1)$. Then, 
    \begin{equation*}
        \gamma^\varepsilon\to \gamma\qquad \text{in }C^k([0,1]).
    \end{equation*}
\end{lemma}

\begin{proof}
By assumptions, we have that $\gamma^\epsilon(s)\to \gamma(s)$ for every $s\in[0,1]$ and $\{f=1\}, \{f_\epsilon=1\}$ are compact. Thus, it suffices to show the local $C^k$-convergence.
    Let $v\in\{f=1\}$. Since $1$ is a regular value for $f$, $\nabla f(v)\neq 0$, and w.l.o.g we may assume that $v$ is such that $\partial_yf(v)\neq 0$. Let $O\subset\R^2$ be a neighborhood of $v$ such that $\partial_yf\neq 0$. Since $\nabla f_\varepsilon\to\nabla f$ uniformly, for $\varepsilon$ sufficiently small $\partial_yf_\varepsilon\neq 0$ on $O$ as well. For every $\varepsilon$, let $v_\varepsilon\in O\cap\{f_\varepsilon=1\}$, such that $v_\varepsilon\to v$. With the natural identification $v=v_0$, for every $\varepsilon\geq 0$, there exist open sets $U_{v_\varepsilon},V_{v_\varepsilon}\subset \R$ and a function $g_\varepsilon:U_{v_\varepsilon}\subset\R\to \R$ such that $v_\varepsilon\in U_{v_\varepsilon}\times V_{v_\varepsilon}\subset O$,  
    \begin{equation*}
        \{f_\varepsilon=1\}\cap U_{v_\varepsilon}\times V_{v_\varepsilon}=\{(t,g_\varepsilon(t)):t\in U_{v_\varepsilon}\},
    \end{equation*}
    and 
    \begin{equation}
    \label{eq:der_g}
        g_\varepsilon'(t)=-\frac{\partial_xf_\varepsilon}{\partial_yf_\varepsilon}(t,g_\varepsilon(t)).
    \end{equation}
Moreover, we may choose $U_{v_\varepsilon}=U_{v}$ and $V_{v_\varepsilon}=V_v$ and, up to an affine reparametrization, we may assume that $v_\varepsilon=(0,g_\varepsilon(0))$. Clearly $g_\varepsilon\to g$ uniformly on $U_v$. 
Thanks to \eqref{eq:der_g}, $g_\varepsilon\to g$ in $C^1(U_v)$ and a simple iterative argument, using \eqref{eq:der_g}, shows that $g_\varepsilon\to g$ in $C^k(U_v)$. At this point, consider $\eta_\varepsilon:U_v\to\R^2$; $\eta_\varepsilon(t):=(t,g_\varepsilon(t))$. Then, there exists $I_v\subset [0,1]$ such that $\gamma^\varepsilon(I_v)=\eta_\varepsilon(U_v)$, up to possibly restricting $U_v$. Then, we can find $\tilde I_v=[0,\alpha]$ and $\tilde\gamma^\varepsilon:\tilde I_v\to \R^2$ for every $\epsilon\geq 0$, such that $\tilde\gamma^\varepsilon:\tilde I_v\to \R^2$ is a $C^k$ arc-length (affine) reparametrization of a sub-arc of ${\gamma^\varepsilon}_{|I_v}$. By further restricting $U_v$ to $\tilde{U}_v$, and defining the $C^k$-diffeomorphic embedding  $\phi_\varepsilon: \tilde{U}_v\to \R$ by
    \begin{equation*}
        \phi_\varepsilon(t):=\int_0^t|\dot\eta_\varepsilon(\tau)|\de\tau
    \end{equation*}
    we have $\phi_\varepsilon\to \phi_0$ in $C^k(\tilde{U}_v)$. Furthermore, note that $\tilde I_v\subset \bigcap_{\epsilon>0} \phi_\epsilon(\tilde{U}_v)$ and $\dot{\phi}_\epsilon$ is uniformly bounded below for small $\epsilon>0$ on $\tilde{U}_v$. Thus, we have $\varphi_\epsilon^{-1}\to \varphi^{-1}$ in $C^k(\tilde I_v)$.
Thus, since $\eta_\varepsilon=\tilde{\gamma}^\varepsilon\circ \phi_\varepsilon$, we conclude that $\tilde{\gamma}^\varepsilon\to\tilde{\gamma}$ in $C^k(\mathcal{I}_v)$, and also $\gamma_\epsilon\to\gamma$ in $C^k_{loc}$ since their reparametrizations are affine.
\end{proof}

\begin{lemma}
\label{lem:P_convergence}
    Let $\{\normdot_\varepsilon\}_{\varepsilon >0}$ be a family of $C^k$ and strongly convex norms on $\R^2$, which $C^k$-strongly converges to a $C^k$ and strongly convex norm $\normdot_0$ on $\R^2$ as $\varepsilon \to 0$. For every $\varepsilon\geq 0$, let $P^\varepsilon : \R/2\pi_{\Omega_\varepsilon}\mathbb{Z} \to \partial\Omega_\varepsilon\subset \R^2$ that associate to the generalized angle $\theta$ the vector $P^\varepsilon_\theta$. Then, 
    \begin{equation}
    \label{eq:Ck_convergence}
        P^\varepsilon\to P^0, \qquad \text{in }C^k_{loc}(\R,\R^2),
    \end{equation}
    extending the maps by periodicity. 
\end{lemma}

\begin{proof}
    For every $\varepsilon\geq 0$, let $\gamma^\varepsilon:[0,1]\to \R^2$ be a constant-speed $C^k$ parametrization of $\partial\Omega_\varepsilon$. Assume that $\gamma^\varepsilon(0)=\gamma^\varepsilon(1)=P^\varepsilon(0)$ for every $\varepsilon\geq 0$. Then, as a consequence of Lemma \ref{lem:curvescurvescurves}, $\gamma^\varepsilon\to\gamma^0$ in the $C^k$-topology. Reasoning as in Lemma \ref{lem:regularity_P}, for every $\varepsilon\geq 0$, we find a $C^k$-diffeomorphism $\theta_\varepsilon:[0,1]\to [0,2\pi_{\Omega_\varepsilon}]$ such that 
    \begin{equation}
    \label{eq:parametrization}
        \gamma^\varepsilon(t)=P^\varepsilon(\theta_\varepsilon(t)),\qquad\forall\,t\in[0,1]. 
    \end{equation}
    Since $\theta_\varepsilon$ is defined by \eqref{eq:def_theta}, with $\gamma^\varepsilon$ in place of $\gamma$ in the left-hand side, we deduce that $\theta_\varepsilon\to\theta_0$ in $C^k([0,1])$. Now, for every $\varepsilon\geq 0$, consider the function $\varphi_\varepsilon:[0,1]\to [0,1]$ defined by $\varphi_\varepsilon(t):=\theta_\varepsilon^{-1}(2\pi_{\Omega_\varepsilon}t)$. Since $2\pi_{\Omega_\varepsilon}\to 2\pi_{\Omega_0}$ and $\theta_\varepsilon\to\theta_0$ in $C^k([0,1])$, we have that $\varphi_\varepsilon\to \varphi_0$ in $C^k([0,1])$, implying \eqref{eq:Ck_convergence} as $P^\varepsilon(2\pi_{\Omega_\varepsilon}t)=\gamma^\varepsilon(\varphi_\varepsilon(t))$, for every $t\in [0,1]$.
\end{proof}

\begin{corollary}
\label{cor:C_convergence}
    Let $\{\normdot_\varepsilon\}_{\varepsilon >0}$ be a family of $C^k$ and strongly convex norms on $\R^2$, which $C^k$ strongly converges to a $C^k$ and strongly convex norm $\normdot_0$ on $\R^2$ as $\varepsilon \to 0$. Then, 
    \begin{equation}
        C_\circ^\varepsilon\to C_\circ^0,\qquad\text{in }C_{loc}^{k-1}(\R,\R)\text{ as $\varepsilon \to 0$}, 
    \end{equation}
    extending the maps by periodicity. 
\end{corollary}

\begin{proof}
    Recall that, for every $\varepsilon\geq 0$, $C_\circ^\varepsilon=(P^\varepsilon)^{-1}\circ\mathrm{d}(\normdot_\varepsilon)_*\circ Q^\varepsilon$ by \eqref{eq:Ccirctonorm}, and observe that, by Lemma \ref{lem:regularity_P}, $C_\circ^\varepsilon$ is a $C^{k-1}$ function. Now, thanks to our assumption and by Lemma \ref{lem:P_convergence}, $\mathrm{d}(\normdot_\varepsilon)_*\circ Q^\varepsilon$ converges to $\mathrm{d}(\normdot_0)_*\circ Q^0$ as $\varepsilon\to 0$ in the $C^{k-1}$-topology. Thus, we are left to check that $(P^\varepsilon)^{-1}\to (P^0)^{-1}$ as $\varepsilon \to 0$ in the $C^k$-topology. Define $f_\varepsilon:\R^2\setminus\{0\}\to \R$ as
    \begin{equation}
        f_\varepsilon(x):=(P^\varepsilon)^{-1}\left(\frac{x}{\norm{x}_\varepsilon}\right),\qquad\forall\,x\in\R^2\setminus\{0\},\ \varepsilon\geq 0.
    \end{equation}
    We claim that $f_\varepsilon\to f$ in $C^k_{loc}(\R^2\setminus\{0\})$ as $\varepsilon\to 0$. Indeed, let $K\subset\R^2\setminus\{0\}$ be compact, then for $\varepsilon_0$ sufficiently small the family $\{f_\varepsilon\}_{\varepsilon\in (0,\varepsilon_0)}$ is uniformly bounded in $K$:
    \begin{equation}  
        |f_\varepsilon(x)|\leq 2\pi_{\Omega_\varepsilon}\leq 4\pi_{\Omega_0}, \qquad\forall\,x\in K.
    \end{equation}
    In addition, it is equi-Lipschitz on $K$, as for $x,y\in K$ and $\varepsilon\in (0,\varepsilon_0)$ we have 
    \begin{equation}
        |f_\varepsilon(x)-f_\varepsilon(y)|\leq \Lip((P^\varepsilon)^{-1})\left|\frac{x}{\norm{x}_\varepsilon}-\frac{y}{\norm{y}_\varepsilon}\right|_{\R^2}\leq \Lip((P^\varepsilon)^{-1})\sup_{z\in K} \mathrm{d}_z\left(\frac{\cdot}{\normdot_\varepsilon}\right)|x-y|_{\R^2},
    \end{equation}
    and the constant in the right-hand side is uniformly bounded since $\{\normdot_\varepsilon\}_{\varepsilon >0}$ $C^k$-strongly converges to $\normdot_0$ as $\varepsilon\to 0$ and $\Lip((P^\varepsilon)^{-1})$ is bounded by the Euclidean distance between $\partial\Omega^\varepsilon$ and the origin, cf.\ \cite[Lem.\ 2.8]{borza2024measure}. Then, given any sequence $\{\varepsilon_n\}_{n \in \N}$ such that $\varepsilon_n \to 0$ as $n\to \infty$, using Arzel\`a-Ascoli theorem we can extract a subsequence $\{\varepsilon_{n(m)}\}_{m \in \N}$ such that $\{f_{\varepsilon_{n(m)}}\}_{m \in \N}$ converges uniformly as $m \to \infty$. Thanks to Lemma \ref{lem:P_convergence}, we must have that $\{f_{\varepsilon_{n(m)}}\}_{m \in \N} \to f_0$ uniformly as $m \to \infty$. Computing the differential of $f_{\varepsilon_{n(m)}}$, we see that
    \begin{equation}
        \mathrm d_xf_{\varepsilon_{n(m)}}=\frac{1}{(P^{\varepsilon_{n(m)}})'\circ f_{\varepsilon_{n(m)}}}\,\mathrm{d}_x\left(\frac{\cdot}{\normdot_{\varepsilon_{n(m)}}}\right).
    \end{equation}
    Since $(P^{\varepsilon_{n(m)}})'$ is uniformly bounded away from $0$ and converges uniformly to $(P^0)'$, we deduce that $\mathrm d_xf_{\varepsilon_{n(m)}}\to \mathrm d_xf_0$ on $K$ as $m\to \infty$, implying that $f_{\varepsilon_{n(m)}}\to f_0$ in $C^1_{loc}(\R^2\setminus\{0\})$ as $m\to \infty$. With an inductive argument, we conclude that $(P^{\varepsilon_{n(m)}})^{-1}\to (P^0)^{-1}$ as $m \to \infty$ in $C^k_{loc}(\R^2\setminus\{0\})$. As this holds for every sequence $\{\varepsilon_n\}_{n \in \N}$, we deduce that $(P^\varepsilon)^{-1}\to (P^0)^{-1}$ in $C^k_{loc}(\R^2\setminus\{0\})$ as $\varepsilon \to 0$, concluding the proof.
\end{proof}

\begin{theorem}
\label{thm:continuity}
Let $k\geq 3$ and let $\{\normdot_\varepsilon\}_{\varepsilon >0}$ be a family of $C^k$ and strongly convex norms on $\R^2$, which $C^k$-strongly converges to a $C^k$ and strongly convex norm $\normdot_0$ on $\R^2$ as $\varepsilon\to 0$. For every $\varepsilon\geq 0$, denote by $\di_\varepsilon$ the \sF distance induced by $\normdot_\varepsilon$ on $\hei$ and by $N_{\mathrm{curv}}^\varepsilon$ the curvature exponent of $(\hei,\di_\varepsilon,\Leb^3)$. Then, we have that
        \begin{equation}
    N_{\mathrm{curv}}^0=\lim_{\varepsilon\to 0} N_{\mathrm{curv}}^\varepsilon.
        \end{equation}     
\end{theorem}

\begin{proof}
     In order to prove the continuity of the curvature exponent, we need to address the convergence of the reduced Jacobians associated with the norms. To this aim, it is convenient to introduce the following parametrization: for every $\varepsilon\geq 0$, define the map
    \begin{equation}
        \widetilde\J_R^\varepsilon:[0,1)\times (0,1)\to [0,\infty); \qquad \widetilde\J_R^\varepsilon(s,r):=\J_R^\varepsilon(2\pi_{\Omega^\circ_\varepsilon}s,-2\pi_{\Omega^\circ_\varepsilon}(1-r)+2\pi_{\Omega^\circ_\varepsilon}r).
    \end{equation}
    Then, recalling the explicit expression of the reduced Jacobian \eqref{eq:reduced_Jac}, as a consequence of Lemma \ref{lem:P_convergence} and Corollary \ref{cor:C_convergence}, we easily obtain that
    \begin{equation}
    \label{eq:uniform_convergence_red_jactilde}
        \widetilde \J_R^\varepsilon\xrightarrow{\varepsilon\to 0}\widetilde \J_R^0,\qquad\text{in }C^{k-1}([0,1)\times (0,1)).
    \end{equation}
    In addition, since the norms of the sequence are $C^1$ and strongly convex, \cite[Prop.\ 3.16]{borza2024measure} implies that the reduced Jacobians are always non-negative and 
    \begin{equation}
    \label{eq:zeros_redjac}
        \lim_{r\to r^\star}\widetilde \J_R^\varepsilon(s,r)=0\qquad\Longleftrightarrow\qquad r^\star=0,1,\frac12.
    \end{equation}
    According to Proposition \ref{prop:diffcharacMCP}, after repara\-metrization, we have that 
    \begin{equation}
    \label{eq:tha_sup}
    N_{\mathrm{curv}}^\varepsilon = \sup_{(s,r)\in [0,1)\times (0,1)} \widetilde N^\varepsilon(s,r)
    \end{equation}
    where, for every $(s,r)\in [0,1)\times (0,1)$,
    \begin{equation}
        \widetilde N^\varepsilon(s,r):=1+\frac{(r-\frac12)\partial_r \widetilde\J_R^\varepsilon(s,r)}{\widetilde\J_R^\varepsilon(s,r)}. 
    \end{equation}
    Firstly, observe that for every $s\in [0,1)$, $\J_R^0(s,\cdot)$ is strictly positive on $(0,1)\setminus\{\frac12\}$ and \eqref{eq:zeros_redjac} holds. Therefore, there exists $\delta>0$ sufficiently small such that $(r-\frac12)\partial_r\J_R^0(s,r)<0$ in $[0,1)\times ((0,\delta)\cup (1-\delta,1))$. Since $\partial_r \widetilde\J_R^\varepsilon(s,r)\to\partial_r \widetilde\J_R^0(s,r)$ uniformly as $\varepsilon\to 0$, then there exists $\varepsilon_0>0$ such that 
    \begin{equation}
        \Big(r-\frac12\Big)\partial_r \widetilde\J_R^\varepsilon(s,r)<0,\qquad\forall\, (s,r)\in [0,1)\times ((0,\delta)\cup (1-\delta,1)),\,\varepsilon< \varepsilon_0.
    \end{equation}
    This means that the supremum in \eqref{eq:tha_sup} can be taken on the set $[0,1)\times [\delta,1-\delta]$. Secondly, we claim that $\widetilde N^\varepsilon\to \widetilde N^0$ uniformly on $[0,1)\times [\delta,1-\delta]$, this would conclude the proof.
    
    To prove the claim, first of all, by \eqref{eq:uniform_convergence_red_jactilde} and \eqref{eq:zeros_redjac}, we immediately deduce that $\widetilde N^\varepsilon\to \widetilde N^0$ locally uniformly on $[0,1)\times ((0,1)\setminus \{\frac12\})$. Second of all, by periodicity and compactness, the convergence is actually uniform as $s\in [0,1)$, therefore we are left to check that we have a uniform convergence around $r=\frac12$. To do so, we use \eqref{eq:JRaround0} to obtain a simplified expression for $N^\varepsilon$: after reparametrization, we get the following:
    \begin{equation}
    \label{eq:Ntilde_epsilon}
        \widetilde N^\varepsilon(s,r)=1+ \frac{(r-\frac12)\partial_r(I^\varepsilon(s,r)+\widetilde R^\varepsilon(s,r))}{I^\varepsilon(s,r)+\widetilde R^\varepsilon(s,r)}=1+\frac{(r-\frac12)\partial_rI^\varepsilon(s,r)}{I^\varepsilon(s,r)}+\mathcal R^\varepsilon(s,r), 
    \end{equation}
    where $\widetilde R^\varepsilon$ denotes the reparametrization of $R^\varepsilon$ defined in \eqref{eq:remainderterm} (for the norm $\normdot_\varepsilon$) and 
    \begin{equation}
    \label{eq:def_I}
        I^\varepsilon(s,r):=\frac{1}{2} \int_{1/2}^{r} \left( \int_{1/2}^{r} (t - \tau)^2 (C_\circ^\varepsilon)'(2\pi_{\Omega^\circ_\varepsilon}(s-1+2t)) (C_\circ^\varepsilon)'(2\pi_{\Omega^\circ_\varepsilon}(s-1+2\tau)) \diff \tau \right) \diff t.
    \end{equation}
    In \eqref{eq:Ntilde_epsilon}, the term $\mathcal R^\varepsilon(s,r)$ is defined as the following difference 
    \begin{equation}
        \widetilde N^\varepsilon(s,r)-1-\frac{(r-\frac12)\partial_rI^\varepsilon(s,r)}{I^\varepsilon(s,r)}.
    \end{equation}
    The proof now follows from Lemma \ref{lem:convergence_quotient} and Lemma \ref{lem:convergence_remainder}.
\end{proof}

\begin{lemma}
\label{lem:convergence_quotient}
    With the assumptions and notations of Theorem \ref{thm:continuity}, denoting for every $\varepsilon\geq 0$
    \begin{equation}
        \mathcal Q^\varepsilon(s,r):=\frac{(r-\frac12)\partial_rI^\varepsilon(s,r)}{I^\varepsilon(s,r)},\qquad\forall\,(s,r)\in [0,1)\times (0,1), 
    \end{equation}
    where $I^\varepsilon$ is defined in \eqref{eq:def_I}, we have $\mathcal Q^\varepsilon\to \mathcal Q^0$ uniformly on $[0,1)\times (0,1)$, as $\varepsilon\to 0$.
\end{lemma}

\begin{proof}
    Before proving the claim, let us observe that 
    \begin{equation}
    \label{eq:expression_I_eps}
        \partial_rI^\varepsilon(s,r)=(C_\circ^\varepsilon)'(2\pi_{\Omega^\circ_\varepsilon}(s-1+2r))\underbrace{\int_{1/2}^{r} (t - r)^2 (C_\circ^\varepsilon)'(2\pi_{\Omega^\circ_\varepsilon}(s-1+2t)) \de t}_{=:J^\varepsilon(s,r)}.
    \end{equation}
    The function $J^\varepsilon(s,r)$ has a Taylor expansion with uniform (in $\varepsilon$) remainder at $r=\frac12$, indeed: 
    \begin{equation}
    \label{eq:expansion_J_eps}
    \begin{split}
        J^\varepsilon(s,r) &=\frac{(C_\circ^\varepsilon)'(2\pi_{\Omega^\circ_\varepsilon}s)}{3}\Big(r-\frac12\Big)^3+\int_{1/2}^r \frac{(t-\frac12)^3}{3}(C_\circ^\varepsilon)''(2\pi_{\Omega^\circ_\varepsilon}(s-1+2t))\de t\\
                           &=\frac{(C_\circ^\varepsilon)'(2\pi_{\Omega^\circ_\varepsilon}s)}{3}\Big(r-\frac12\Big)^3+O\left(\Big(r-\frac12\Big)^4\right),\qquad\text{as }r\to\frac12,
    \end{split}
    \end{equation}
    where the big-$O$ term is uniform in $\varepsilon$ and depends on a uniform bound for $\Lip((C_\circ^\varepsilon)'')$. From this expansion, we deduce an expansion for the function $I^\varepsilon(s,r)$ at $r=\frac12$, using \eqref{eq:expression_I_eps} and a uniform bound on $\Lip((C_\circ^\varepsilon)')$. More precisely, we have: 
    \begin{equation}
    \label{eq:expansion_I_eps}
        I^\varepsilon(s,r)=\int_{1/2}^r \partial_rI^\varepsilon(s,\tau)\de\tau=\frac{(C_\circ^\varepsilon)'(2\pi_{\Omega^\circ_\varepsilon}s)^2}{12}\Big(r-\frac12\Big)^4+O\left(\Big(r-\frac12\Big)^5\right),\qquad\text{as }r\to\frac12,
    \end{equation}
    where the big-$O$ term is once again uniform in $\varepsilon$.
    
    We intend to apply Arzel\`a-Ascoli theorem to the functions $\mathcal{Q}^\varepsilon$. Firstly, by \eqref{eq:expansion_I_eps} and \eqref{eq:expression_I_eps}, there exists constants $C>c>0$, not depending on $\varepsilon$, such that
    \begin{equation}
    \label{eq:uniform_bds_I}
        c \Big|r-\frac12\Big|^4\leq |I^\varepsilon|\leq C \Big|r-\frac12\Big|^4,\qquad |\partial_rI^\varepsilon|\leq C\Big|r-\frac12\Big|^3.
    \end{equation}
    This means that the functions $\mathcal{Q}^\varepsilon$ are uniformly bounded by a constant not depending on $\varepsilon$. Secondly, we find a uniform bound for the gradient of $\mathcal Q^\varepsilon$. For the derivative in the $s$-direction, the estimates from the first part of the proof, together with a uniform bound on $\Lip((C_\circ^\varepsilon)')$, allows to conclude. For the derivative in the $r$-direction, it is convenient to write:
    \begin{multline}
    \label{eq:expression_Q_eps}
        \partial_r \mathcal Q^\varepsilon(s,r)=4\pi_{\Omega^\circ_\varepsilon}  (C_\circ^\varepsilon)''(2\pi_{\Omega^\circ_\varepsilon}(s-1+2r)) \frac{(r-\frac12)J^\varepsilon(s,r)}{I^\varepsilon(s,r)}\\+ (C_\circ^\varepsilon)'(2\pi_{\Omega^\circ_\varepsilon}(s-1+2r))\partial_r\left( \frac{(r-\frac12)J^\varepsilon(s,r)}{I^\varepsilon(s,r)}\right)
    \end{multline}
    The first term can be uniformly estimated reasoning as in the first part of the proof, using the expansion \eqref{eq:expansion_J_eps} and also a uniform bound on $\Lip((C_\circ^\varepsilon)')$. For controlling the second term, we need a uniform bound on 
    \begin{equation}
        \partial_r\left( \frac{(r-\frac12)J^\varepsilon(s,r)}{I^\varepsilon(s,r)}\right)=\frac{(J^\varepsilon(s,r)+(r-\frac12)\partial_rJ^\varepsilon(s,r))I^\varepsilon(s,r)-(r-\frac12)J^\varepsilon(s,r)\partial_r I^\varepsilon(s,r)}{I^\varepsilon(s,r)^2}.
    \end{equation}
    Note that, a priori, the numerator has order $7$ at $r=\frac12$, while the denominator has order $8$. However, replacing the Taylor expansions \eqref{eq:expansion_I_eps} and \eqref{eq:expansion_J_eps}, after a routine computation, one obtains that: 
    \begin{equation}
        (J^\varepsilon(s,r)+(r-\frac12)\partial_rJ^\varepsilon(s,r))I^\varepsilon(s,r)-(r-\frac12)J^\varepsilon(s,r)\partial_r I^\varepsilon(s,r)=O\left(\Big(r-\frac12\Big)^8\right),\qquad\text{as }r\to\frac12
    \end{equation}
    where the big-$O$ is uniform in $\varepsilon$. This shows the claimed uniform bound for the second term of \eqref{eq:expression_Q_eps}. Finally, we have proven that $\{\mathcal Q^\varepsilon\}_{\varepsilon\geq 0}$ is uniformly bounded and equi-Lipschitz as $\varepsilon\to 0$. Therefore, since we also know that $\mathcal Q^\varepsilon(s,r)\to\mathcal Q^0(s,r)$ pointwise as $\varepsilon\to 0$, we conclude the proof of the lemma. 
\end{proof}

\begin{lemma}
\label{lem:convergence_remainder}
    With the assumptions and notations of Theorem \ref{thm:continuity}, $\mathcal R^\varepsilon\to\mathcal R^0$ uniformly on $[0,1)\times (0,1)$, as $\varepsilon\to 0$.
\end{lemma}

\begin{proof}
Firstly observe that, for the norm $\normdot_\varepsilon$, the remainder term \eqref{eq:remainderterm}, denoted by $R^\varepsilon$, can be bounded by $M_\varepsilon \omega^6$ around $\omega=0$, where $M_\varepsilon$ depends on the $L^\infty$-norms of $P^\varepsilon$, $Q^\varepsilon$ and $(C_\circ^\varepsilon)'$ which are all uniformly bounded in $\varepsilon$. Indeed: 
    \begin{equation}
    \begin{split}
        |R^\varepsilon(\varphi,\omega)| &\leq \Lip(C_\circ^\varepsilon)^3\int_\varphi^{\varphi+\omega}\int_\varphi^t\int_\varphi^s (t-s)(s-u)[2\|Q^\varepsilon\|_{L^\infty}\|P^\varepsilon\|_{L^\infty}(u-\varphi)+(t-\varphi)]\de u \de s \de t\\
        &\leq 2\Lip(C_\circ^\varepsilon)^3(2\|Q^\varepsilon\|_{L^\infty}\|P^\varepsilon\|_{L^\infty}+1)\int_\varphi^{\varphi+\omega}\int_\varphi^t\int_\varphi^s (t-s)(s-u)(t-\varphi)\de u \de s \de t\\
        &\leq 2\Lip(C_\circ^\varepsilon)^3(2\|Q^\varepsilon\|_{L^\infty}\|P^\varepsilon\|_{L^\infty}+1)\int_\varphi^{\varphi+\omega}\int_\varphi^t\int_\varphi^s (t-s)(s-u)(t-\varphi)\de u \de s \de t\\
        &\leq \frac{1}{72}\Lip(C_\circ^\varepsilon)^3(2\|Q^\varepsilon\|_{L^\infty}\|P^\varepsilon\|_{L^\infty}+1)\omega^6, \qquad\forall\, \varphi,\omega.
    \end{split}
    \end{equation}
    Hence, after reparametrization, there exists $M_1>0$ such that 
    \begin{equation}
    \label{eq:contazzi1}
        |\widetilde R^\varepsilon(s,r)|\leq M_1\left|r-\frac{1}{2}\right|^6,\qquad\text{on } [0,1)\times(0,1).
    \end{equation}
    Reasoning as above, one can obtain a similar estimate for the $r$-derivative, namely there exists $M_2>0$ such that
    \begin{equation}
    \label{eq:contazzi2}
        |\partial_r\widetilde R^\varepsilon(s,r)|\leq M_2\left|r-\frac{1}{2}\right|^5,\qquad\text{on } [0,1)\times(0,1).
    \end{equation}
    Now, by definition of $\mathcal R^\varepsilon(s,r)$, we see that 
    \begin{equation}
        \mathcal R^\varepsilon(s,r)=\frac{(r-\frac12)\left(I^\varepsilon(s,r)\partial_r\widetilde R^\varepsilon(s,r)-\widetilde R^\varepsilon(s,r)\partial_rI^\varepsilon(s,r))\right)}{I^\varepsilon(s,r)(I^\varepsilon(s,r)+\widetilde R^\varepsilon(s,r))}.
    \end{equation}
    The right-hand side of the above can be uniformly bounded combining \eqref{eq:contazzi1}, \eqref{eq:contazzi2} and \eqref{eq:uniform_bds_I}. Thus, we conclude that there exists a uniform constant $M_3>0$ such that 
    \begin{equation}
        |\mathcal R^\varepsilon(s,r)|\leq M_3\left|r-\frac{1}{2}\right|^2,\qquad\text{on } [0,1)\times(0,1).
    \end{equation}
    With completely similar arguments, one can show that there exits a constant $M_4>0$, depending on the uniform bounds of the $L^\infty$-norms of $P^\varepsilon$, $Q^\varepsilon$, $(C_\circ^\varepsilon)'$ and $(C_\circ^\varepsilon)''$, such that 
     \begin{equation}
        |\partial_r\mathcal R^\varepsilon(s,r)|\leq M_4\left|r-\frac{1}{2}\right|,\qquad\text{on } [0,1)\times(0,1).
    \end{equation}
    Thus, applying Arzel\`a-Ascoli, we conclude the proof.
\end{proof}

\begin{theorem}\label{thm:continuouscurvexp}
    Given any $N^\star>5$ and $k \in (2,+\infty)$, there exists a $C^k$ and strongly convex norm $\normdot$ such that the curvature exponent $N_\mathrm{curv}$ of the associated \sF Heisenberg group $(\hei,\di,\Leb^3)$ is equal to $N^\star$.
\end{theorem}

\begin{proof}

    Let $p\in (1,2)$ and let $q\in (2,\infty)$ be its conjugate exponent. Choose $p$ sufficiently close to $1$ so that $q\geq \min\{k,3\}$ and $2q+1\geq N^\star$. Consider the curve of norms 
    \begin{equation}
    \label{eq:construction}
        [0,1]\ni t\mapsto \normdot^t:=f^t_*,\qquad\text{with }f^t:=t\, \ell^q+(1-t)\ell^2,
    \end{equation}
    where $\ell^q$ denotes the standard $\ell^q$-norm and $\ell^2$ is the classical Euclidean norm. Then, for every $t\in [0,1)$, $f^t$ is strongly convex and of class $C^{\lfloor q \rfloor}$ and $t\mapsto f^t$ is a continuous curve of norms with respect to the topology of $C^{\lfloor q \rfloor}$-strongly convergence (with $\lfloor q \rfloor\geq \min\{k,3\}$). Therefore, according to Remark \ref{rmk:dual_is_bello}, $t\mapsto \normdot^t$ is a continuous curve of strongly convex and $C^{\lfloor q \rfloor}$ norms on $[0,1)$. Moreover, the endpoints of $t\mapsto\normdot^t$ are respectively the $\ell^2$-norm at $t=0$ and the $\ell^p$-norm at $t=1$. By Theorem \ref{thm:continuity} (which applies since $\lfloor q \rfloor\geq 3$) and by Lemma \ref{lem:lsc_curv_exp}, we have
    \begin{equation}
    \label{eq:crazy_continuities}
        \lim_{t\to t_0}N_{\mathrm{curv}}^t=N_{\mathrm{curv}}^{t_0},\quad\forall\, t_0\in (0,1)\qquad\text{and}\qquad  N_
{\mathrm{curv}}^1\leq\liminf_{t\to 1}N_{\mathrm{curv}}^t.
    \end{equation}
    But now, since at time $t=1$, $\normdot^1=\ell^p$, by \cite[Thm.\ A]{borzatashiro}, $N_{\mathrm{curv}}^1>2q+1\geq N^\star$. Hence, from \eqref{eq:crazy_continuities}, we have that $t\mapsto\normdot^t$ is a curve of norms whose curvature exponents connect continuously $N_{\mathrm{curv}}^0=5$ and $2q+1$. Thus, there exists $t_0$ such that $N_{\mathrm{curv}}^{t_0}=N^\star$, and setting $\normdot:=\normdot^{t_0}$, we find a $C^k$ and strongly convex norm with the desired curvature exponent.
    \qedhere
\end{proof}

\begin{remark}\label{rmk:arbitrarybignorm}
    Given any $N^\star\geq 5$, it is possible to construct a $C^\infty$ and strongly convex norm such that the corresponding \sF Heisenberg group $(\hei,\di,\Leb^3)$ has $N_\mathrm{curv} \geq N^\star$. Here we outline a possible construction. For an integer $h\geq 3$, let us fix a function 
    \[
    g(y)=1-y^2-\frac{h^h}{2}y^h.
    \]
    Note that for $y\in[0,\frac{1}{h}]$,
    \[
    g(y)>0,\quad g'(y)\leq 0,\quad g^{\prime\prime}(y)<0,
    \]
    so there is a $C^\infty$ and strongly convex norm $\normdot_{\ast }^h$ such that its unit sphere locally coincides with the level set $x-g(y)=0$ for $y\in[0,\frac{1}{h}]$. 
    For each $y\in[0,\frac{1}{h}]$, we can associate the generalized angle $\omega(y)$ so that $(g(y),y)=(\cosomp(\omega(y)),\sinomp(\omega(y)))$.
We can recover the functions $\omega$, $\sinom(\omega_{\circ}),\cosom(\omega_{\circ})$ and $\omega_\circ$ with the function $g(y)$ by using their geometric characterization, see Figure \ref{fig:convextrig1} and \ref{fig:convextrig2}.
Combined with $C_\circ^\prime(\omega)=\frac{\diff \omega_\circ}{\diff y}\frac{\diff y}{\diff \omega}$,
we can write
\begin{align*}
    \J_R(0,\omega)|_{\omega=\omega(y)}={}&[2-\cosomp(\omega)-\cosom(\omega_\circ)-\omega\sinom(\omega)]|_{\omega=\omega(y)}\\
    ={}&2-g(y)-\frac{1}{g(y)-yg'(y)}-\left(2\int_0^yg(t)\diff t-yg(y)\right)\frac{-g'(y)}{g(y)-yg'(y)},
\end{align*}
and
\begin{align*}
    \omega\partial_\omega\J_R(0,\omega)|_{\omega=\omega(y)}={}&\omega C_\circ^\prime(\omega)\left[\sinomp(\omega)-\omega\cosomp(\omega)\right]|_{\omega=\omega(y)}\\
    ={}&\left(2\int_0^yg(t)\diff t-yg(y)\right)\frac{-g^{\prime\prime}(y)}{g(y)-yg'(y)}\left[y-\left(2\int_0^yg(t)\diff t-yg(y)\right)g(y).\right]
\end{align*}
    One can show that, for every $h\in\N$, the angle $\omega_h:=\omega(1/h)$ (i.e. $y=1/h$) satisfies 
    \begin{equation}
        \label{eq:finalratio}
        \frac{\omega_h\partial_\omega\J_R(0,\omega_h)}{\J_R(0,\omega_h)}=\frac{h}{4}+R(h),
    \end{equation}
    where $\omega_h\to 0$ and $R(h)/h\to 0$ as $h\to \infty$. Then, the ratio \eqref{eq:finalratio} diverges to $+\infty$ as $h\to+\infty$ and the conclusion follows from Proposition \ref{prop:diffcharacMCP}. 
\end{remark}

This remark allows us to improve the regularity of norms in Theorem \ref{thm:continuouscurvexp} and proves Theorem \ref{thm:INTROcontinuity}.

\begin{theorem}\label{thm:continuouscurvexpSMOOOOOOTH}
    Given any $N^\star>5$, there exists a $C^\infty$ and strongly convex norm $\normdot$ such that the curvature exponent $N_\mathrm{curv}$ of the associated \sF Heisenberg group $(\hei,\di,\Leb^3)$ is equal to $N^\star$.
\end{theorem}

\begin{proof}
    We repeat the same argument of Theorem \ref{thm:continuouscurvexp},
    replacing the $\ell^q$-norm in \eqref{eq:construction} with the smooth and strongly convex norm $\normdot^h_{*}$ of Remark \ref{rmk:arbitrarybignorm}, with $h$ sufficiently big. Note that the norm $\normdot$, found with the construction, is $C^\infty$ and strongly convex, as it is the dual of a $C^{\infty}$ and strongly convex norm, cf.\ Remark \ref{rmk:dual_is_bello}. 
\end{proof}


\bibliographystyle{alpha} 
\bibliography{bibliography}

\end{document}